\documentclass[11pt]{amsart}
\usepackage{float}
\usepackage[usenames]{color}
\usepackage{graphicx}
\usepackage{amssymb}
\usepackage{amscd}
\theoremstyle{plain}
\newtheorem{thm}{Theorem}[section]

\newtheorem{cor}[thm]{Corollary}

\theoremstyle{definition}

\newtheorem{rmk}[thm]{Remark}
\newtheorem{example}[thm]{Example}

\def\dim{\mathop{\hbox {dim}}\nolimits}

\newcommand{\fra}{\mathfrak{a}}

\newcommand{\frg}{\mathfrak{g}}
\newcommand{\frh}{\mathfrak{h}}

\newcommand{\frk}{\mathfrak{k}}
\newcommand{\frl}{\mathfrak{l}}

\newcommand{\frp}{\mathfrak{p}}
\newcommand{\frq}{\mathfrak{q}}

\newcommand{\frt}{\mathfrak{t}}
\newcommand{\fru}{\mathfrak{u}}

\newcommand{\bbC}{\mathbb{C}}

\newcommand{\bbN}{\mathbb{N}}

\newcommand{\bbR}{\mathbb{R}}

\newcommand{\bbZ}{\mathbb{Z}}

\newcommand{\caL}{\mathcal{L}}

\newcommand{\caR}{\mathcal{R}}

\newcommand{\be}{\begin {equation}}
\newcommand{\ee}{\end {equation}}
\newcommand{\bp}{\begin {proof}}
\newcommand{\ep}{\end {proof}}

\begin{document}

\title[Dirac series for some real exceptional Lie groups]
{Dirac series for some real exceptional Lie groups}

\author{Jian Ding}
\address[Ding]{School of Mathematics, Hunan University, Changsha 410082,
P.~R.~China}
\email{dingjain@hnu.edu.cn}

\author{Chao-Ping Dong}
\address[Dong] {Mathematics and Science College, Shanghai Normal University, Shanghai 200234, P.~R.~China }
%{School of Mathematical Sciences, Soochow University, Suzhou 215006, P.~R.~China}
\email{chaopindong@163.com}
%\thanks{Dong is supported by NSFC grant 11571097 and Shanghai Gaofeng Project for University Academic Development Program.}

\author{Liang Yang}
\address[Yang]{College of Mathematics, Sichuan University, Chengdu 610064,
P.~R.~China}
\email{malyang@scu.edu.cn}

\abstract{Up to equivalence, this paper classifies all the irreducible unitary representations with non-zero Dirac cohomology for the following simple real exceptional Lie groups: ${\rm EI}=E_{6(6)}, {\rm EIV}=E_{6(-26)}, {\rm FI}=F_{4(4)}, {\rm FII}=F_{4(-20)}$. Along the way, we find an irreducible unitary representation of $F_{4(4)}$ whose Dirac index vanishes, while its Dirac cohomology is non-zero. This disproves a conjecture raised in 2015 asserting that there should be no cancellation between the even part and the odd part of the Dirac cohomology.}
 \endabstract

\subjclass[2010]{Primary 22E46.}

\keywords{Dirac cohomology, Dirac index, elliptic representation, unipotent representation, unitary representation.}

\maketitle
\section{Introduction}

Let $G(\bbC)$ be a complex connected  simple algebraic group with finite center.
Let $\sigma: G(\bbC) \to G(\bbC)$ be a \emph{real form} of $G(\bbC)$. That is, $\sigma$ is an antiholomorphic Lie group automorphism and $\sigma^2={\rm Id}$. Let $\theta: G(\bbC)\to G(\bbC)$ be the involutive algebraic automorphism of $G(\bbC)$ corresponding to $\sigma$ via Cartan theorem (see Theorem 3.2 of \cite{ALTV}). Put $G=G(\bbC)^{\sigma}$ as the group of real points. Note that  $G$ must be in the Harish-Chandra class \cite{HC}.  Denote by $K(\bbC):=G(\bbC)^{\theta}$, and put $K:=K(\bbC)^{\sigma}$. Denote by $\frg_0$ the Lie algebra of $G$, and let $\frg_0=\frk_0\oplus \frp_0$ be the  Cartan decomposition corresponding to $\theta$ on the Lie algebra level. Denote by $\frh_{f, 0}=\frt_{f, 0}\oplus \fra_{f, 0}$ the unique $\theta$-stable fundamental Cartan subalgebra of $\frg_0$. That is, $\frt_{f, 0}\subseteq \frk_0$ is maximal abelian.   As usual, we drop the subscripts to stand for the complexified Lie algebras. For example, $\frg=\frg_0\otimes_{\bbR}\bbC$, $\frh_f=\frh_{f, 0}\otimes_{\bbR}\bbC$ and so on. We fix a non-degenerate invariant symmetric bilinear form $B$ on $\frg$. Its restrictions to $\frk$, $\frp$, etc., will also be denoted by the same symbol.

The current paper aims to classify $\widehat{G}^d$---the set of equivalence classes of irreducible unitary $(\frg, K)$-modules with non-zero Dirac cohomology---for some real exceptional Lie groups. After Huang, we call $\widehat{G}^d$ the \emph{Dirac series} of $G$.  Among the entire unitary dual $\widehat{G}$, the Dirac series are precisely the ones where Parthasarathy's Dirac operator inequality \cite{P1,P2} becomes equality on some $K$-types. See  Section \ref{sec-Dirac} for more. Thus understanding these modules thoroughly should be very interesting. Indeed, the problem of classifying $\widehat{G}^d$ remains open ever since Huang and Pand\v zi\'c's proof \cite{HP} of the Vogan conjecture in 2002.

One tool for us is Theorem A of \cite{D17}, which says that $\widehat{G}^d$ consists of the scattered part (finitely many scattered representations) and the string part (finitely many strings of representations). Moreover, the string part of $\widehat{G}^d$ comes from the scattered part of $\widehat{L}^d$ via cohomological induction, where $L$ runs over the Levi factors of the finitely many $\theta$-stable  parabolic subgroups of $G$. Recall that a member of $\widehat{G}^d$ is \textbf{scattered} if it is not cohomologically induced from any irreducible unitary module in the good range from any proper $\theta$-stable parabolic subgroup of $G$. This theorem allows us to completely determine $\widehat{G}^d$ after a finite calculation and to organize it neatly for those $G$ whose rank is small. Another tool is the software \texttt{atlas} \cite{At}, which detects unitarity based on the algorithm due to Adams, van Leeuwen, Trapa and Vogan in \cite{ALTV}.

In this paper, we will classify the Dirac series for the following groups:
$$
{\rm EI}=E_{6(6)}, \quad {\rm EIV}=E_{6(-26)}, \quad
{\rm FI}=F_{4(4)}, \quad {\rm FII}=F_{4(-20)}.
$$
In the statement of the following results, by the \textbf{FS-scattered part} of $\widehat{G}^d$ we mean the members of $\widehat{G}^d$ whose KGB elements have full support. This is a subset of the scattered part of $\widehat{G}^d$. See the end of Section \ref{sec-algorithm} for an explanation, and Section \ref{sec-G} for examples.

\medskip
\noindent\textbf{Theorem A.}
\emph{The set $\widehat{\rm FII}^{\mathrm{d}}$ consists of two FS-scattered representations (see Table \ref{table-FII-scattered-part}) whose spin-lowest $K$-types are all unitarily small, and ten strings of representations (see Table \ref{table-FII-string-part}). Moreover, each  representation $\pi\in\widehat{\rm FII}^{\mathrm{d}}$ has a unique spin-lowest $K$-type occurring with multiplicity one.}
\medskip

The notion spin-lowest $K$-type will be given in Section \ref{sec-lambda-spin}, and that of unitarily small (\emph{u-small} for short henceforth) was introduced by Salamanca-Riba and Vogan \cite{SV}, see \eqref{u-small}.

\medskip
\noindent\textbf{Theorem B.}
\emph{The set $\widehat{\rm EIV}^{\mathrm{d}}$ consists of two FS-scattered representations (see Table \ref{table-EIV-scattered-part}) whose spin-lowest $K$-types are all u-small, and nine strings of representations (see Table \ref{table-EIV-string-part}). Moreover, each  representation $\pi\in\widehat{\rm EIV}^{\mathrm{d}}$ has a unique spin-lowest $K$-type occurring with multiplicity one.}
\medskip

\medskip
\noindent\textbf{Theorem C.}
\emph{The set $\widehat{\rm FI}^{\mathrm{d}}$ consists of twenty two FS-scattered representations (see Table \ref{table-FI-scattered-part}) whose spin-lowest $K$-types are all u-small, and seventy six strings of representations (see Tables \ref{table-FI-string-part-card0}---\ref{table-FI-string-part-card3}). Moreover, each spin-lowest $K$-type of any  representation $\pi\in\widehat{\rm FI}^{\mathrm{d}}$ occurs with multiplicity one.}
\medskip

It is interesting to note that members of $\widehat{\rm FI}^{\mathrm{d}}$ could have more than one spin-lowest $K$-types. This phenomenon was observed earlier by Barbasch and Pand\v zi\'c on some classical groups \cite{BP15}. A careful look at the Dirac cohomology of one FS-scattered representation of FI allows us to disprove Conjecture 10.3 of \cite{H15}. See Example \ref{exam-counter}. It proves that even for Dirac series of equal rank groups, the Dirac index (see \cite{MPV}) can vanish.

\medskip
\noindent\textbf{Theorem D.}
\emph{The set $\widehat{\rm EI}^{\mathrm{d}}$ consists of thirteen FS-scattered representations (see Table \ref{table-EI-scattered-part}) whose spin-lowest $K$-types are all u-small, and forty three  strings of representations (see Tables \ref{table-EI-string-part-card0}---\ref{table-EI-string-part-card5}). Moreover, each spin-lowest $K$-type of any  representation $\pi\in\widehat{\rm FI}^{\mathrm{d}}$ occurs with multiplicity one.}
\medskip

It is interesting to note that distinct scattered members of $\widehat{\rm EI}^d$ can have the same spin-lowest $K$-types (hence the same Dirac cohomology). See Tables \ref{table-EI-scattered-part} and \ref{table-G-scattered-part}. We also note that a spin-lowest $K$-type of an arbitrary irreducible unitary representation can have multiplicity greater than one, see Example \ref{exam-spin-mult-big}.

Our calculations suggest that the number of FS-scattered members of $\widehat{G}^d$ (denoted by $N_{\rm FS}$) is related to the number of u-small $K$-types (denoted by $N_{\rm us}$). A brief summary is given in Table \ref{FS-us}.

A very interesting phenomenon that we must mention is that certain FS-scattered members can be viewed as the limit cases of some strings. See Example \ref{exam-G-scattered-via-string}.

\begin{table}
\centering
\caption{$N_{\rm FS}$ and $N_{\rm us}$}
\begin{tabular}{c|c|c|c|c|c}
& $E_{6(-26)}$ &   $E_{6(6)}$   & $F_{4(-20)}$ & $F_{4(4)}$ & $G_{2(2)}$  \\
\hline
$N_{\rm FS}$ & $2$ & $13$ & $2$ & $22$ & $3$\\
$N_{\rm us}$ & $37$ & $484$ & $27$ & $544$ & $16$
\end{tabular}
\label{FS-us}
\end{table}

In the spirit of Conjecture 1.1 of Barbasch and Pand\v zi\'c \cite{BP11}, the scattered part of $\widehat{G}^d$ should offer us some unipotent representations---the most singular unitary representations which are believed to be the building blocks of $\widehat{G}$. Therefore, our classifications should also be helpful for understanding the entire unitary dual of the relevant exceptional Lie groups.

The paper is organized as follows. We recall necessary background in Section 2. Then we adopt  the algorithm for computing the scattered part of $\widehat{G}^d$ from \cite{D17} in Section 3. Based on it, classifications of $\widehat{\rm FII}^d$, $\widehat{\rm EIV}^d$, $\widehat{\rm FI}^d$ and $\widehat{\rm EI}^d$ are reported in Sections 4-7.
Section 8 illustrates the toy case $\widehat{G_{2(2)}}^d$ carefully.  Note that the unitary dual of $G_{2(2)}$ has been determined by Vogan \cite{Vog94} in 1994.

The root systems are adopted as in Appendix C of Knapp \cite{Kn} throughout the paper. We always use $\bbN$ to stand for the set $\{0,1,2,\dots\}$.

\section{Preliminaries}
This section aims to  collect necessary preliminaries. We adopt the basic notation $G$, $\theta$, $K$ etc.,  as in the introduction. Let $H_f=T_f A_f$ be the $\theta$-stable fundamental Cartan subgroup for $G$. Then  $T_f$ is a maximal torus of $K$.
 Put $\frh_f=\frt_f+\fra_f$ for the Cartan decomposition  on the complexified Lie algebras level.
Recall that a non-degenerate invariant symmetric bilinear form $B$ has been fixed on $\frg$. Its restrictions to $\frk$, $\frp$, etc., will also be denoted by $B$.

We denote by $\Delta(\frg, \frh_f)$ (resp.,  $\Delta(\frg, \frt_f)$) the root system of $\frg$ with respect to $\frh_f$ (resp., $\frt_f$).
The root system of $\frk$ with respect to $\frt_f$ is denoted by $\Delta(\frk, \frt_f)$.
Note that $\Delta(\frg,
\frh_f)$ and $\Delta(\frk, \frt_f)$ are reduced, while $\Delta(\frg, \frt_f)$ is not reduced in general. The corresponding Weyl groups are written as $W(\frg, \frh_f)$, $W(\frg, \frt_f)$ and $W(\frk, \frt_f)$.

We fix compatible choices of positive roots $\Delta^+(\frg, \frh_f)$ and $\Delta^+(\frg, \frt_f)$ so that a positive root in $\Delta(\frg, \frh_f)$ restricts to a positive root in $\Delta(\frg, \frt_f)$. Note that $\Delta^+(\frg, \frt_f)$ is a union of the set of positive compact roots $\Delta^+(\frk, \frt_f)$ and the set of positive noncompact roots  $\Delta^+(\frp, \frt_f)$.
As usual, we denote by  $\rho$ (resp. $\rho_c$, $\rho_n$) the half sum of roots in $\Delta^+(\frg, \frh_f)$ (resp. $\Delta^+(\frk, \frt_f)$, $\Delta^+(\frp, \frt_f)$). Then $\rho, \rho_c, \rho_n\in i\frt_{f, 0}^*$ and $\rho_n=\rho-\rho_c$.

\subsection{$\texttt{atlas}$ height, lambda norm and spin norm}\label{sec-lambda-spin}
We will simply refer to a $\frk$-type by  its highest weight $\mu$. Choose a positive root system $(\Delta^+)^\prime(\frg, \frh_f)$ making $\mu+2\rho_c$ dominant. Let $\rho^\prime$ be the half sum of roots in $(\Delta^+)^\prime(\frg, \frh_f)$. After \cite{SV}, we define $\lambda_a(\mu)$ to be the projection of $\mu+2\rho_c- \rho^\prime$ to the dominant Weyl chamber of $(\Delta^+)^\prime(\frg, \frh_f)$. Then
\begin{equation}\label{lambda-norm}
\|\mu\|_{\rm{lambda}}:=\|\lambda_a(\mu)\|.
\end{equation}
Here $\|\cdot\|$ is the norm induced from the form $B$. The number \eqref{lambda-norm}  is  independent of the choice of  $(\Delta^+)^\prime(\frg, \frh_f)$, and it is the \emph{lambda norm} of the $\frk$-type $\mu$ \cite{Vog81}. It is worth noting that the trivial $\frk$-type could have non-zero lambda norm. Now the \texttt{atlas} \emph{height} of $\mu$ can be computed by
\begin{equation}\label{atlas-height}
\sum_{\alpha\in (\Delta^+)^\prime(\frg, \frh_f)}\langle \lambda_a(\mu),  \alpha^{\vee}\rangle.
\end{equation}

For the fixed $\Delta^+(\frk, \frt_f)$, let us enumerate all the compatible choices of positive roots for $\Delta(\frp, \frt_f)$ as
$$
(\Delta^+)^{(j)}(\frp, \frt_f), \quad 0\leq j\leq s-1.
$$
Here $(\Delta^+)^{(0)}(\frp, \frt_f)=\Delta^+(\frp, \frt_f)$ and
$$
s=\frac{|W(\frg, \frt_f)|}{|W(\frk, \frt_f)|}.
$$
Denote by $\rho_n^{(j)}$ the half sum of roots in $(\Delta^+)^{(j)}(\frp, \frt_f)$. In particular, $\rho_n^{(0)}=\rho_n$. Now the \emph{spin norm} of the $\frk$-type $\mu$ \cite{D} is defined as
\begin{equation}\label{spin-norm}
\|\mu\|_{\rm{spin}}:=\min_{0\leq j\leq s-1} \|\{\mu-\rho_n^{(j)}\}+\rho_c\|,
\end{equation}
where $\{\mu-\rho_n^{(j)}\}$ stands for the unique dominant weight to which $\mu-\rho_n^{(j)}$ is conjugate under the action of $W(\frk, \frt_f)$.
 Note that $\{\mu-\rho_n^{(j)}\}$ is a \emph{PRV component} \cite{PRV} of the tensor product of the $\frk$-type $\mu$ and the spin module ${\rm Spin}_G$.
 For any $\frk$-type $\mu$, it is shown in \cite{D} that
\begin{equation}\label{lambda-spin-norm}
\|\mu\|_{\rm spin}\geq \|\mu\|_{\rm lambda}.
\end{equation}
Since $G$ is in the Harish-Chandra class, we may and we will define the lambda norm (resp., spin norm) of a $K$-type as
the lambda norm (resp., spin norm) of any of its highest weights. Now a $K$-type of a $(\frg, K)$-module $\pi$ is called a \emph{lambda-lowest} (resp., \emph{spin-lowest}) \emph{$K$-type} if its lambda norm (resp. spin norm) attains the minimum among all the $K$-types of $\pi$. Theorem 1.2 of \cite{DD16}  characterizes those $K$-types  whose spin norm is equal to their lambda norm.

Finally, recall from \cite{SV} that a $\frk$-type $\mu$ is called  \emph{u-small} if $\mu$ lies in the following convex hull
\begin{equation}\label{u-small}
\left\{   \sum_{\alpha\in\Delta(\frp, \frt_f)} b_{\alpha} \alpha\mid 0\leq b_{\alpha}\leq 1\right\}.
\end{equation}

\subsection{Dirac cohomology}\label{sec-Dirac}

Fix
an orthonormal basis $Z_1, \dots, Z_n$ of $\frp_0$ with respect to
the inner product induced by the form $B$. Let $U(\frg)$ be the
universal enveloping algebra of $\frg$ and let $C(\frp)$ be the
Clifford algebra of $\frp$ with respect to $B$. Living in $U(\frg)\otimes C(\frp)$, the Dirac operator is defined as
$$D=\sum_{i=1}^{n}\, Z_i \otimes Z_i.$$
It is easy to check that $D$ does not depend on the choice of the
orthonormal basis $Z_i$ and it is $K$-invariant for the diagonal
action of $K$ given by adjoint actions on both factors.

To achieve a better understanding of the unitary dual $\widehat{G}$, Vogan introduced  Dirac cohomology in 1997 \cite{Vog97}. Indeed, let $\pi$ be a
($\frg$, $K$)-module and choose a spin module  ${\rm Spin}_G$  for $C(\frp)$.
Then $D\in U(\frg)\otimes C(\frp)$ acts on $\pi\otimes {\rm Spin}_G$, and the Dirac
cohomology of $\pi$ is defined as the $\widetilde{K}$-module
\begin{equation}\label{def-Dirac-cohomology}
H_D(\pi)=\text{Ker}\, D/ (\text{Im} \, D \cap \text{Ker} D).
\end{equation}
Here $\widetilde{K}$ is the subgroup of $K\times \text{Pin}\,\frp_0$ consisting of all pairs $(k, s)$ such that $\text{Ad}(k)=p(s)$, where $\text{Ad}: K\rightarrow \text{O}(\frp_0)$ is the adjoint action, and $p: \text{Pin}\,\frp_0\rightarrow \text{O}(\frp_0)$ is the pin double covering map. Namely, $\widetilde{K}$ is constructed from the following diagram:
\[
\begin{CD}
\widetilde{K} @>  >  > {\rm Pin}\, \frp_0 \\
@VVV  @VVpV \\
K @>{\rm Ad}>> O(\frp_0)
\end{CD}
\]
Note that $\widetilde{K}$ acts  on $\pi$
through $K$ and on ${\rm Spin}_G$ through the pin group
$\text{Pin}\,{\frp_0}$. Moreover, since ${\rm Ad}(k) (Z_1), \dots, {\rm Ad}(k) (Z_n)$ is still an orthonormal basis of $\frp_0$, it follows that $D$ is $\widetilde{K}$ invariant. Therefore, ${\rm  Ker} D$, ${\rm Im} D$, and $H_D(X)$ are $\widetilde{K}$ modules.

 We embed $\frt_f^{*}$ as a subspace of $\frh_f^{*}$ by setting  the linear functionals on $\frt_f$ to be zero on $\fra_f$. Huang and Pand\v zi\'c proved the  Vogan conjecture in Theorem 2.3 of \cite{HP}. Here we recall a slight extension of this result to possibly disconnected Lie groups.

\begin{thm}{\rm (Theorem A \cite{DH})}\label{thm-HP}
Let $\pi$ be an irreducible ($\frg$, $K$)-module.
Assume that the Dirac
cohomology of $\pi$ is nonzero, and let $\gamma\in\frt_f^{*}\subset\frh_f^{*}$ be any highest weight of a $\widetilde{K}$-type  in $H_D(X)$. Then the infinitesimal character $\Lambda$ of $\pi$ is conjugate to
$\gamma+\rho_{c}$ under $W(\frg,\frh_f)$.
\end{thm}

Guaranteed by the above theorem, a necessary condition for $\pi$ to have non-zero Dirac cohomology is that $\Lambda$ is \emph{real} in the sense of Definition 5.4.11 of \cite{Vog81}. That is, $\Lambda\in i\frt_{f,0}^*+\fra_{f,0}^*$.

We care the most about the case that $\pi$ is unitary. Then $D$ is self-adjoint with respect to a natural inner product on $\pi\otimes {\rm Spin}_G$, $\text{Ker} D \cap \text{Im} D=0$, and
\begin{equation}\label{Dirac-unitary}
H_D(\pi)=\text{Ker}\, D=\text{Ker}\, D^2.
\end{equation}
Moreover, $D^2$ has non-negative eigenvalue on any $\widetilde{K}$-type of $\pi\otimes {\rm Spin}_G$. Recall from Proposition 3.1.6 of \cite{HP2} that
$$
D^2=\Omega_{\frk_{\Delta}} -\Omega_{\frg}\otimes 1 +
(\|\rho_c\|^2-\|\rho\|^2) 1 \otimes 1.
$$
Then we are led to Parthasarathy's Dirac operator inequality \cite{P1, P2}
$$
\|\gamma+\rho_c\| \geq \|\Lambda\|,
$$
where $\gamma$ is any highest weight of any $\widetilde{K}$-type in $X\otimes {\rm Spin}_G$.
Taking the PRV-components \cite{PRV} of the tensor products of the $K$-types and ${\rm Spin}_G$ into account, this can be encapsulated as
\begin{equation}\label{Dirac-inequality}
\|\mu\|_{\rm spin}\geq \|\Lambda\|,
\end{equation}
where $\mu$ is a highest weight of any $K$-type occurring in $\pi$.
Moreover, in view of Theorem 3.5.2 of \cite{HP2}, \eqref{Dirac-inequality} becomes equality on certain $K$-types of $\pi$ if and only if $H_D(\pi)$ is non-vanishing.

\subsection{Cohomological induction}

Let $\frq= \frl\oplus\fru$ be  a $\theta$-stable parabolic subalgebra of $\frg$ with Levi factor $\frl$ and nilpotent radical $\fru$. Set $L=N_{G}(\frq)$.

We arrange the positive root systems so that
$$
\Delta(\fru, \frh_f)\subseteq \Delta^{+}(\frg,\frh_f),
\quad \Delta^{+}(\frl, \frh_f)=\Delta(\frl,
\frh_f)\cap \Delta^{+}(\frg,\frh_f).
$$
Let $\rho^{L}$  denote the half sum of roots in
$\Delta^{+}(\frl,\frh_f)$,  and denote by $\rho(\fru)$ (resp., $\rho(\fru\cap\frp)$) the half sum of roots in $\Delta(\fru,\frh_f)$ (resp., $\Delta(\fru\cap\frp,\frh_f)$). Then
\begin{equation}\label{relations}
\rho=\rho^{L}+\rho(\fru).
\end{equation}

Let $Z$ be an ($\frl$, $L\cap K$) module. Cohomological induction functors lift $Z$ to certain ($\frg, K$)-modules $\caL_j(Z)$ and $\caR^j(Z)$, where $j$ is a nonnegative integer.
Suppose that $Z$ has infinitesimal character $\lambda_L\in \frh_f^*$. As in \cite{KV}, we say that
$Z$ is {\it good} or {\it in good range} if
\begin{equation}\label{def-good}
\mathrm{Re}\,\langle \lambda_L +\rho(\fru),\, \alpha^\vee \rangle >
0, \quad \forall \alpha\in \Delta(\fru, \frh_f).
 \end{equation} We say that $Z$ is {\it weakly good} if
\begin{equation}\label{def-weakly-good}
\mathrm{Re}\, \langle \lambda_L +\rho(\fru),\, \alpha^\vee \rangle
\geq 0, \quad \forall \alpha\in \Delta(\fru, \frh_f).
\end{equation}

\begin{thm}\label{thm-Vogan-coho-ind}
{\rm (\cite{Vog84} Theorems 1.2 and 1.3, or \cite{KV} Theorems 0.50 and 0.51)}
Suppose the admissible
 ($\frl$, $L\cap K$)-module $Z$ is weakly good.  Then we have
\begin{itemize}
\item[(i)] $\caL_j(Z)=\caR^j(Z)=0$ for $j\neq S(:=\emph{\text{dim}}\,(\fru\cap\frk))$.
\item[(ii)] $\caL_S(Z)\cong\caR^S(Z)$ as ($\frg$, $K$)-modules.
\item[(iii)]  if $Z$ is irreducible, then $\caL_S(Z)$ is either zero or an
irreducible ($\frg$, $K$)-module with infinitesimal character $\lambda_L+\rho(\fru)$.
\item[(iv)]
if $Z$ is unitary, then $\caL_S(Z)$, if nonzero, is a unitary ($\frg$, $K$)-module.
\item[(v)] if $Z$ is in good range, then $\caL_S(Z)$ is nonzero, and it is unitary if and only if $Z$ is unitary.
\end{itemize}
\end{thm}

Assume the weight $\Lambda\in\frh_f^*$ is dominant for $\Delta^+(\frg, \frh_f)$.
We say $\Lambda$ is  {\it strongly regular} if
\begin{equation}\label{def-weakly-good}
\langle \Lambda-\rho,\, \alpha^\vee \rangle
\geq 0, \quad \forall \alpha\in \Delta^+(\frg, \frh_f).
\end{equation}

\begin{thm}\label{thm-SR} \emph{(Salamanca-Riba \cite{Sa})}
Let $X$ be an irreducible $(\frg, K)$-module with a strongly regular real infinitesimal character. If $X$ is unitary, then it is an $A_{\frq}(\lambda)$ module in the good range.
\end{thm}

Note that \cite[Theorem B]{DH} gives a formula of Dirac cohomology for weakly good cohomologically induced modules.

%\begin{cor}\label{cor-thm-DH}
%Suppose that the irreducible unitary $(\frl, L\cap K)$-module $Z$ has real infinitesimal character $\lambda_L\in i\frt_{f, 0}^*$ which is weakly good. Then $H_D(\caL_S(Z))$ is non-zero if and only if that $H_D(Z)$ is non-zero, and that there exists a highest weight $\gamma_L$ in $H_D(Z)$ such that $\gamma_L +\rho(\fru\cap\frp)$ is $\Delta^+(\frk, \frt_f)$ dominant.
%\end{cor}

\subsection{Cohomological induction in the software \texttt{atlas}}\label{sec-coho-atlas}
Let us warm up with some basic setting of the software \texttt{atlas} \cite{At}. Let $H(\bbC)$ be a \emph{maximal torus} of $G(\bbC)$. That is, $H(\bbC)$ is a maximal connected abelian subgroup of $G(\bbC)$ consisting of diagonalizable matrices.  Its \emph{character lattice} is the group of algebraic homomorphisms
$$
X^*:={\rm Hom}_{\rm alg} (H(\bbC), \bbC^{\times}).
$$
Choose a Borel subgroup $B(\bbC)\supset H(\bbC)$.
In \texttt{atlas}, an irreducible $(\frg, K)$-module $\pi$ is parameterized by a \emph{final} parameter $p=(x, \lambda, \nu)$ via the Langlands classification \cite{ALTV}, where $x$ is a $K(\bbC)$-orbit of the Borel variety $G(\bbC)/B(\bbC)$, $\lambda \in X^*+\rho$ and $\nu \in (X^*)^{-\theta}\otimes_{\bbZ}\bbC$. In such a case, the infinitesimal character of $\pi$ is
\begin{equation}\label{inf-char}
\frac{1}{2}(1+\theta)\lambda +\nu \in\frh^*,
\end{equation}
where $\frh$ is the Lie algebra of $H(\bbC)$.
Note that the Cartan involution $\theta$ now becomes $\theta_x$---the involution of $x$, which is given by the command \texttt{involution(x)} in \texttt{atlas}.

Taken from Paul's lecture \cite{Paul}, the following result expresses a theorem of Vogan \cite{Vog84} in the language of \texttt{atlas}.

\begin{thm}\label{thm-Vogan} \emph{(Vogan \cite{Vog84})}
 Let $p=(x, \lambda, \nu)$ be the \texttt{atlas} parameter of an irreducible $(\frg, K)$-module $X$.
Let $S$ be the support of $x$, and $\frq(x)$ be the $\theta$-stable parabolic subalgebra given by the pair $(S, x)$, with Levi factor $L$.
Then  $X$ is cohomologically induced, in the weakly good range,
from an irreducible $(\frl, L\cap K)$-module $X_L$ with parameter $p_L=(y, \lambda-\rho(\fru), \nu)$, where $y$ is the KGB element of $L$ corresponding to the KGB element $x$ of $G$.
\end{thm}

Note that the \emph{support} of a KGB element $x$ (that is, a $K(\bbC)$-orbit of the Borel variety $G(\bbC)/B(\bbC)$) is given by the command \texttt{support(x)} in \texttt{atlas}. The above theorem always utilizes the minimum $\theta$-stable parabolic subgroup to do cohomological induction.  Usually, one can find other bigger $\theta$-stable parabolic subgroups to realize $X$ via cohomological induction, then we may move the inducing module  from the weakly good range to the good range.

\begin{example}\label{exam-coho-ind}
Let us illustrate Theorem \ref{thm-Vogan} via the representation in the seventh row of Table \ref{table-FII-string-part}. (Due to the different labeling of simple roots of \texttt{atlas}, we should reverse the coordinates of $\lambda$ and $\nu$ there.)  For simplicity, certain outputs of \texttt{atlas} have been omitted here.
\begin{verbatim}
G:F4_B4
set p=parameter(KGB(G,6), [1,1,1,0], [-3/2,3/2,0,-3/2])
set (P,pL)=reduce_good_range(p)
P
Value: ([1,2],KGB element #6)
rho_u(P)
Value: [ 5, 0, 0, 6 ]/2
pL
Value: final parameter(x=3,lambda=[-3,2,2,-6]/2,nu=[-3,3,0,-3]/2)
theta_induce_irreducible(pL, G)=p
Value: true
goodness(pL,G)
Value: "Weakly good"
\end{verbatim}
The \texttt{pL} above is the parameter of the inducing module. The last output says that it is weakly good.

Now let us enlarge the above $\theta$-stable parabolic subgroup \texttt{P} to be the following \texttt{Q}, then we can still get the original representation \texttt{p}, while the inducing module will be shifted to the good range.
\begin{verbatim}
set Q=theta_stable_parabolics(G)[30]
Q
Value: ([1,2,3],KGB element #9)
set L=Levi(Q)
L
Value: connected real group with Lie algebra 'sp(2,1).u(1)'
rho_u(Q)
Value: [ 4, 0, 0, 0 ]/1
set qL=parameter(KGB(L,6),[1,1,1,0]-rho_u(Q), [-3/2,3/2,0,-3/2])
qL
Value: final parameter(x=6,lambda=[-3,1,1,0]/1,nu=[-3,3,0,-3]/2)
theta_induce_irreducible(qL,G)=p
Value: true
goodness(qL,G)
Value: "Good"
\end{verbatim}
Therefore, the representation \texttt{p} is actually \emph{not} a scattered member of $\widehat{\rm FII}^d$.
\hfill\qed
\end{example}

\section{Computing the scattered part of $\widehat{G}^d$}\label{sec-algorithm}
This section aims to recall the algorithm for computing the scattered part of $\widehat{G}^d$ from \cite{D17}. Note that a member of $\widehat{G}^d$ belongs to the scattered part if it is \emph{not} cohomologically induced from any \textbf{good} module of any \textbf{proper} $\theta$-stable parabolic subgroup of $G$.

Based on the proof of \cite[Theorem A]{D17}, we proceed as follows:
\begin{itemize}
\item[(a)] Enumerate all the dominant real infinitesimal characters $\Lambda$ such that
$$
\|\Lambda\|^2<\left(\min_{\alpha\in\Delta^+(\frg, \frh)}\frac{4\|\rho\|^2}{\|\alpha\|^2}+1\right)\|\rho\|^2,
$$
and such that $\Lambda$ is conjugate to $\delta+\rho_c$ for a certain $K$-type $\delta$.
\item[(b)] For each $\Lambda$  in step (a), enumerate all the irreducible representations of $G$ with infinitesimal character $\Lambda$ via the command
\begin{verbatim}
set all=all_parameters_gamma(Lambda)
\end{verbatim}
Further select the unitary ones out of the  above modules  via the command
\begin{verbatim}
for p in all do if is_unitary(p) then prints(p) fi od
\end{verbatim}
\item[(c)] For each module $\texttt{p}$ surviving in step (b), check whether its Dirac cohomology vanishes or not. More precisely, given the infinitesimal character $\Lambda$, one enumerates all the $K$-types \texttt{CanK} that can possibly contribute to Dirac cohomology by Theorem \ref{thm-HP}. Calculate the maximum \texttt{atlas} height $\texttt{ht}$ of the $K$-types in \texttt{CanK}. Then look at the $K$-types of $\texttt{p}$ up to this height via the command
\begin{verbatim}
print_branch_irr_long(p, KGB(G,0), ht)
\end{verbatim}
The Dirac cohomology of $\texttt{p}$ is non-vanishing   if and only if at least one $K$-type in \texttt{CanK} occurs in the output of the above command.
\end{itemize}

If the group $G$ is connected and centerless, then any one-dimensional unitary character of $G$ must be trivial. Therefore, whenever this is the case,  Theorem \ref{thm-SR} allows us to focus on those $\Lambda$ which are \emph{not} strongly regular in step (a) to find the non-trivial scattered members of $\widehat{G}^d$. This will significantly reduce the workload.

Another remark is that we can use Parthasarathy's Dirac operator inequality to detect non-unitarity effectively in step (b). More precisely, as guaranteed by Theorem C of \cite{D17}, for certain groups we can use the distribution of spin norm along Vogan pencils starting from one of the lowest $K$-types to rule out many non-unitary representations. See Example \ref{exam-EI-Dirac-inequality}.

Carrying out these steps will give us finitely many members of $\widehat{G}^d$, among which we can find all the scattered members due to Theorem A of \cite{D17}. For instance, a representation must be scattered if its KGB element has a full support. There may also be some scattered members whose KGB elements are not fully supported. In practice, we find it is more convenient to view them as the starting points of strings instead of singling them out.  See Example \ref{exam-G2-string-starting-point}.

Note that the scattered members of $\widehat{L}^d$  are embedded  in the representations produced by steps (a)---(c) as well, where $L$ runs over the Levi factors of all the proper $\theta$-stable parabolic subgroups of $G$. Thus we obtain some starting points of all the strings of $\widehat{G}^d$ at the same time. Now it remains to figure out the strings. A typical string will be considered in Example \ref{exam-string}. Other strings are similar. Then we will pin down $\widehat{G}^d$ completely.

To sum up, like the case of complex Lie groups \cite{DD, D17E6}, we will organize  the set $\widehat{G}^d$ neatly  according to $|{\rm supp}(x)|$---the cardinality of the support of \texttt{x}. To be concise, we will refer to the members of $\widehat{G}^d$ whose KGB elements have full support as the \textbf{FS-scattered part} of $\widehat{G}^d$ from now on.

In subsequent sections, we will carry out this algorithm for several simple real exceptional Lie groups. Some \texttt{Mathematica} files are built to facilitate the calculations, and the codes there are carefully explained. The one for the group ${\rm EI}=E_{6(6)}$ is available via the link
\begin{verbatim}
https://www.researchgate.net/publication/327741868_EI-ScatteredPart
\end{verbatim}

\section{The set $\widehat{\rm FII}^d$}\label{sec-FII}

This section aims to classify the Dirac series for ${\rm FII}=F_{4(-20)}$, which is realized in $\texttt{atlas}$ via the command
\texttt{G:F4\_B4}. This  equal rank group is  centerless, connected and simply connected.

We adopt the simple roots of $\Delta^+(\frg, \frt_f)$ and  $\Delta^+(\frk, \frt_f)$ as in Knapp \cite[Appendix C]{Kn}.
In particular, its Vogan diagram is presented in Fig.~\ref{Fig-FII-Vogan}, where $\alpha_1, \alpha_2$ are short, while $\alpha_3, \alpha_4$ are long.
Let $\{\xi_1, \xi_2, \xi_3, \xi_4\}$ be the fundamental weights corresponding to $\Delta^+(\frg, \frt_f)$. The simple roots for $\Delta^+(\frk, \frt_f)$ are
$$
\gamma_1=2\alpha_1+2\alpha_2+\alpha_3, \quad \gamma_2=\alpha_4, \quad \gamma_3=\alpha_3, \quad \gamma_4=\alpha_2.
$$
Let $\varpi_1, \dots, \varpi_4$ be the fundamental weights corresponding to $\Delta^+(\frk, \frt_f)$.
We will express the $\texttt{atlas}$ parameters $\lambda$ and $\nu$ in terms of $\xi_1, \dots, \xi_4$, and express highest weights of $K$-types in terms of $\varpi_1, \dots, \varpi_4$. For instance,  the spin-lowest $K$-type in the first row of Table \ref{table-FII-string-part} is the one with highest weight
$$
[b+c, a-1, b-1, c+d]:=(b+c)\varpi_1 + (a-1)\varpi_2 + (b-1)\varpi_3+ (c+d) \varpi_4.
$$

\begin{figure}[H]
\centering
\scalebox{0.6}{\includegraphics{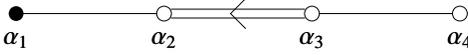}}
\caption{The Vogan diagram for FII}
\label{Fig-FII-Vogan}
\end{figure}

 Note that \texttt{atlas} labels the simple roots of $\Delta^+(\frg, \frt_f)$ in the way \textbf{opposite to} that of Fig.~\ref{Fig-FII-Vogan}. Thus whenever we put the parameters $\lambda$ and $\nu$ in Tables \ref{table-FII-scattered-part} and \ref{table-FII-string-part} into $\texttt{atlas}$, we should \textbf{reverse} the order of their coordinates.

\begin{example}\label{exam-FII-scattered-part}
Let us illustrate the algorithm in Section \ref{sec-algorithm} for FII.
\begin{itemize}
\item[$\bullet$]
Step (a) gives us $311513$ candidates for $\Lambda$, among which $38090$ are not strongly regular. By Theorem \ref{thm-SR}, it suffices to focus on the latter ones.

\item[$\bullet$] FII has $15$ KGB elements in total. The following ones are fully supported
$$
\#x=10, 11, 12, 13, 14.
$$

\item[$\bullet$] Now say fix $\#x=10$. Then only one representation survives after carrying out steps (b) and (c). This gives the first row of Table \ref{table-FII-scattered-part}.
\item[$\bullet$] All other fully supported KGB elements produce no non-trivial FS-scattered members of $\widehat{\rm FII}^d$.
\end{itemize}
To sum up, the FS-scattered members of $\widehat{\rm FII}^d$ are exhausted in Table \ref{table-FII-scattered-part}, where in the second row sits the trivial representation.
See Example \ref{exam-G-scattered-via-string} for the meaning of the column ``string limit".
\hfill\qed
\end{example}

The strings of $\widehat{\rm FII}^d$ are given in Table \ref{table-FII-string-part}, where $a, b, c, d$ are members of $\bbN$ such that
the infinitesimal character
$$
\Lambda=[d, c, b, a]
$$
and that
\begin{equation}\label{FII-common-requirement}
a\geq 1, \quad b\geq 1, \quad c+d\geq 1.
\end{equation}
In a few cases, there are stronger requirements than \eqref{FII-common-requirement} for some coordinates. They will be put within the column ``spin LKTs".

\begin{example}\label{exam-string}
Let us show that each member of the string with $\#x=3$ in Table \ref{table-FII-string-part} is a Dirac series. Note that this KGB element of $\texttt{F4\_B4}$ has support  $[3]$ in \texttt{atlas}.

In view of \eqref{FII-common-requirement}, this string has a unique starting module
$$
\pi_{0, 1, 1}:=(x, [1, 0, 1, 1], [1, -{\textstyle\frac{1}{2}}, 0, 0]).
$$
We compute directly that it is a Dirac series.
It has infinitesimal character $\Lambda_{0, 1, 1}:=[1, 0, 1, 1]=\xi_1+\xi_3+\xi_4$. (Note again that the coordinates of $\lambda$, $\nu$ and $\Lambda$ should be reversed for \texttt{atlas}.) The module $\pi_{0, 1, 1}$ is cohomologically induced from an irreducible unitary module $\pi_{0, 1, 1}^{L}$ (in the way of Theorem \ref{thm-Vogan}) which is weakly good. Fix integers $a\geq 1, b\geq 1, c\geq 0 $. Put $\Lambda_{c, b, a}:=[1, c, b, a]=\xi_1+c\xi_2+b\xi_3+a\xi_4$ and
$$
\pi_{c, b, a}:=(x, [1, c, b, a], [1, -{\textstyle\frac{1}{2}}, 0, 0]).
$$
Let $\zeta_{c, b-1, a-1}$ be the unitary character of $L$ with differential $c\xi_2+(b-1)\xi_3+(a-1)\xi_4$. By our choices of $a, b, c$, we have that
\begin{equation}\label{keeping-weakly-good}
\Lambda_{c, b, a}-\Lambda_{0, 1, 1}=c\xi_2+(b-1)\xi_3+(a-1)\xi_4
\end{equation}
is dominant for $\Delta^+(\frg, \frt_f)$. Therefore, by Theorem 7.237 of \cite{KV},
\begin{equation}\label{Dirac-translation-functor}
\psi_{\Lambda_{c, b, a}}^{\Lambda_{0, 1, 1}}\big(\caL_S(\pi_{0, 1, 1}^L\otimes \zeta_{c, b-1, a-1})\big)=\caL_S\big(\pi_{0, 1, 1}^L\big)=\pi_{0, 1, 1}.
\end{equation}
Here $\psi_{\Lambda_{c, b, a}}^{\Lambda_{0, 1, 1}}$ is the translation functor, and $S:=\dim (\fru\cap\frk)$.
In particular, it says that $\caL_S(\pi_{0, 1, 1}^L\otimes \zeta_{c, b-1, a-1})$ is non-zero.  Note that by \eqref{keeping-weakly-good},
the inducing module $\pi_{0, 1, 1}^L\otimes \zeta_{c, b-1, a-1}$ is weakly good. Therefore, the module $\caL_S(\pi_{0, 1, 1}^L\otimes \zeta_{c, b-1, a-1})$ is irreducible and unitary. This gives us the representation $\pi_{c, b, a}$. Since $\pi_{0, 1, 1}$ has nonzero Dirac cohomology, and that
\begin{equation}\label{Dirac-unitary-character}
H_D\big(\pi_{0, 1, 1}^L\otimes \zeta_{c, b-1, a-1}\big)
=H_D\big(\pi_{0, 1, 1}^L\big)\otimes \zeta_{c, b-1, a-1},
\end{equation}
it follows from Theorem B of \cite{DH} that $\pi_{c, b, a}$ is a Dirac series.
\hfill\qed
\end{example}
\begin{rmk}
The above argument works for all the strings in this paper.
\end{rmk}

One checks directly that each starting point (hence each member) of every string in Table \ref{table-FII-string-part} is either in the good range or can be shifted to the good range by enlarging the $\theta$-stable parabolic subgroups (see Example \ref{exam-coho-ind} for one instance). Therefore, Table \ref{table-FII-scattered-part} actually exhausts the scattered part of $\widehat{\rm FII}^d$.

\begin{cor}
Let $G$ be {\rm FII}. Then all the $K$-types whose spin norm is equal to their lambda norm are exactly
\begin{itemize}
\item[$\bullet$]
$[b+c, a-1, b-1, c+d]$, where $a, b\geq 1$, $c, d\geq 0$ and $c+d\geq 1$;
\item[$\bullet$]
$[b-1, a-1, b+c, d-1]$, where $a, b, d\geq 1$ and $c\geq 0$;
\item[$\bullet$]
$[b+c+d+1, a-1, b-1, c-1]$, where $a, b, c\geq 1$ and $d\geq 0$.
\end{itemize}
\end{cor}
\begin{proof}
Note that the strings with $\#x=0, 1, 2$ give precisely all the irreducible tempered representations of FII with non-zero Dirac cohomology. The result follows from Theorem 1.2 of \cite{DD16} and Table \ref{table-FII-string-part}.
\end{proof}

\begin{table}[H]
\centering
\caption{FS-scattered members of $\widehat{\rm FII}^{\mathrm{d}}$}
\begin{tabular}{l|c|c|c|c|c|r}
$\# x$ &   $\lambda$  & $\nu$ & spin LKTs & mult & u-small & string limit\\
\hline
$10$ & $[2, -1, 1, 2]$ & $[\frac{5}{2}, -\frac{5}{2}, 0, \frac{5}{2}]$ & $[0, 0, 1, 0]$ & $1$ & Yes & $\#8$, $d=-1$\\
$14$ & $[1, 1, 1, 1]$ & $[\frac{11}{2}, 0, 0, 0]$ & $[0, 0, 0, 0]$ & $1$ & Yes
\end{tabular}
\label{table-FII-scattered-part}
\end{table}

\begin{table}[H]
\centering
\caption{Strings of $\widehat{\rm FII}^{\mathrm{d}}$}
\begin{tabular}{l|c|c|c|c|r}
$\# x$ &   $\lambda$   & $\nu$ &spin LKTs & mult  \\
\hline
$0$ & $[d, c, b, a]$ & $[0, 0, 0, 0]$ & $[b+c, a-1, b-1, c+d]$ & $1$ \\
$1$ & $[d, c, b, a]$ & $[0, 0, 0, 0]$ & $[b-1, a-1, b+c, d-1]$, $d\geq 1$ & $1$  \\
$2$ & $[d, c, b, a]$ & $[0, 0, 0, 0]$ & $[b+c+d+1, a-1, b-1, c-1]$, $c\geq 1$ & $1$  \\
$3$ & $[1, c, b, a]$ & $[1, -\frac{1}{2}, 0, 0]$ & $[b+c+1, a-1, b-1, c]$ & $1$  \\
$4$ & $[d, 1, b, a]$ & $[-\frac{1}{2}, 1, -\frac{1}{2}, 0]$ & $[b, a-1, b, d]$ & $1$  \\
$5$ & $[1, 1, b, a]$ & $[1, 1, -1, 0]$  &  $[b+1, a-1, b, 0]$ & $1$  \\
$6$ & $[d, 1, 1, a]$ & $[-\frac{3}{2}, 0, \frac{3}{2}, -\frac{3}{2}]$ & $[0, a, 0, d+1]$ & $1$  \\
$7$ & $[2, -1, 2, a-1]$ & $[\frac{3}{2}, -\frac{3}{2}, \frac{3}{2}, -\frac{3}{2}]$ & $[1, a-1, 0, 1]$ & $1$ \\
$8$ & $[d, 1, 1, 1]$ & $[-\frac{5}{2}, 0, 0, \frac{5}{2}]$ & $[0, 0, 0, d+2]$ & $1$ \\
$9$ & $[1, 1, 1, a]$ & $[0, \frac{5}{2}, 0, -\frac{5}{2}]$ & $[0, a+1, 0, 0]$ & $1$
\end{tabular}
\label{table-FII-string-part}
\end{table}

\section{The set $\widehat{\rm EIV}^d$}

This section aims to classify the Dirac series for ${\rm EIV}=E_{6(-26)}$, which is realized in $\texttt{atlas}$ via the command \texttt{G:E6\_F4}.
This group is centerless, connected and simply connected. It is not equal rank. Indeed, $\dim \frt_f=4$ and $\dim \fra_f=2$.

We adopt the simple roots of $\Delta^+(\frg, \frh_f)$ and  $\Delta^+(\frk, \frt_f)$ as in Knapp \cite[Appendix C]{Kn}.
In particular, its Vogan diagram is presented in Fig.~\ref{Fig-EIV-Vogan}.
Let $\{\xi_1, \dots, \xi_6\}$ be the fundamental weights corresponding to $\Delta^+(\frg, \frt_f)$. The simple roots for $\Delta^+(\frk, \frt_f)$ are
$$
\gamma_4:=\beta_2, \quad \gamma_3:=\beta_4, \quad  \gamma_2:=\frac{1}{2}(\beta_3+\beta_5), \quad \gamma_1:=\frac{1}{2}(\beta_1+\beta_6).
$$
Here $\gamma_1$ and $\gamma_2$ are short, while $\gamma_3$ and $\gamma_4$ are long.
Let $\varpi_1, \dots, \varpi_4$ be the fundamental weights corresponding to $\Delta^+(\frk, \frt_f)$. We will express the highest weights of $K$-types in terms of $\varpi_1, \dots, \varpi_4$. Similarly, the \texttt{atlas} parameters $\lambda$ and $\nu$ will be expressed in terms of $\xi_1, \dots, \xi_6$. Note that \texttt{atlas} labels the simple roots of $\Delta^+(\frg, \frh_f)$ in the same way as that of Fig.~\ref{Fig-EIV-Vogan}.

\begin{figure}[H]
\centering
\scalebox{0.5}{\includegraphics{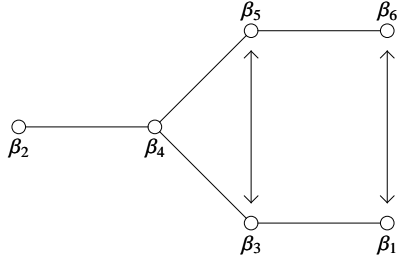}}
\caption{The Vogan diagram for EIV}
\label{Fig-EIV-Vogan}
\end{figure}

\begin{example}\label{exam-EIV-scattered-part}
Let us illustrate the algorithm in Section \ref{sec-algorithm} for EIV.
\begin{itemize}
\item[$\bullet$]
Step (a) gives us $1147419$ candidates for $\Lambda$, among which $105003$ are not strongly regular. By Theorem \ref{thm-SR}, it suffices to focus on the latter ones.

\item[$\bullet$] EIV has $45$ KGB elements in total. The following ones are fully supported
$$
\#x=12, 14, 16, 17, 19, 20, 21, \quad  23\leq \#x\leq 44.
$$

\item[$\bullet$] Now say fix $\#x=19$. Then only one representation survives after carrying out steps (b) and (c). This gives the first row of Table \ref{table-EIV-scattered-part}.
\item[$\bullet$] All other fully supported KGB elements produce no non-trivial FS-scattered members of $\widehat{\rm EIV}^d$.
\end{itemize}
To sum up, the FS-scattered members of $\widehat{\rm EIV}^d$ are exhausted in Table \ref{table-EIV-scattered-part}, where again in the second row sits the trivial representation.
\hfill\qed
\end{example}

The string part of $\widehat{\rm EIV}^d$ is presented in Table  \ref{table-EIV-string-part}, where $a, b, c, d, e, f$ are nonnegative integers such that
the infinitesimal character
$$
\Lambda=[\frac{a+f}{2}, b, \frac{c+e}{2}, d, \frac{c+e}{2}, \frac{a+f}{2}]
$$
and that
\begin{equation}\label{EIV-common-requirement}
a-f=0 \mbox{ or } 1, \quad c- e=0 \mbox{ or } 1, \quad a+f\geq 1, \quad c+e\geq 1, \quad b\geq 1, \quad d\geq 1.
\end{equation}

One checks directly that each starting point (hence each member) of every string in Table \ref{table-EIV-string-part} is in the good range. Therefore, Table \ref{table-EIV-scattered-part} actually exhausts the scattered part of $\widehat{\rm EIV}^d$.

\begin{cor}
Let $G$ be {\rm EIV}. Then all the $K$-types whose spin norm is equal to their lambda norm are exactly $[a, b, c, d]$, where $a, b\geq 1$ and $c, d\geq 0$.
\end{cor}
\begin{proof}
Note that the string with $\#x=0$ gives precisely all the irreducible tempered representations of EIV with non-zero Dirac cohomology. The result follows from Theorem 1.2 of \cite{DD16} and Table \ref{table-EIV-string-part}.
\end{proof}

\begin{table}[H]
\centering
\caption{FS-scattered members of $\widehat{\rm EIV}^{\mathrm{d}}$}
\begin{tabular}{l|c|c|c|c|c|r}
$\# x$ &   $\lambda$  & $\nu$ & spin LKTs & mult & u-small& string limit\\
\hline
$19$ & $[1,2,0,1,0,1]$ & $[\frac{3}{2},3,-\frac{3}{2},0,-\frac{3}{2},\frac{3}{2}]$ & $[1, 1, 0, 0]$ & $1$ & Yes  & $\#7$, $a+f=-1$\\
$44$ & $[1,1,1,1,1,1]$ & $[4,0,0,0,0,4]$ & $[0, 0, 0, 0]$ & $1$ & Yes
\end{tabular}
\label{table-EIV-scattered-part}
\end{table}

\begin{table}[H]
\centering
\caption{Strings of $\widehat{\rm EIV}^{\mathrm{d}}$}
\begin{tabular}{r|c|c|c|c|c}
$\# x$ &   $\lambda$  & $\nu$ & spin LKTs & mult \\
\hline
$0$ & $[a, b, c, d, e, f]$ & $[0, 0, 0, 0, 0, 0]$ & $[a+f, c+e, d-1, b-1]$ & $1$ \\
$1$ & $[a, b, 1, d, 1, f]$ & $[-\frac{1}{2},0,1,-1,1,-\frac{1}{2}]$ & $[a+f+1, 0, d, b-1]$ & $1$ \\
$2$ & $[1, b, c, d, e, 1]$ & $[1,0,-\frac{1}{2},0,-\frac{1}{2},1]$ & $[0, c+e+1, d-1, b-1]$ & $1$\\
$3$ & $[a, b, 1,1,1, f]$ & $[-1,-2,0,2,0,-1]$ & $[a+f+2, 0, 0, b]$& $1$ \\
$7$ & $[a, 1, 1,1,1, f]$ & $[-\frac{3}{2},3,0,0,0,-\frac{3}{2}]$ & $[a+f+3, 0, 0, 0]$& $1$ \\
$9$ & $[1, b, 1, d, 1, 1]$ & $[1,0,1,-2,1,1]$ & $[0, 0, d+1, b-1]$& $1$ \\
$9$ & $[1, b, 1, d-1, 1, 1]$ & $[\frac{1}{2},0,\frac{1}{2},-1,\frac{1}{2},\frac{1}{2}]$ & $[1, 1, d-1, b-1]$& $1$\\
$13$ & $[1, b-1, 0,2, 0, 1]$ & $[1,-2,-1,2,-1,1]$ & $[1, 1, 0, b-1]$& $1$\\
$22$ & $[1, b, 1,1,1, 1]$ & $[0,-4,2,0,2,0]$ & $[0, 0, 0, b+2]$& $1$
\end{tabular}
\label{table-EIV-string-part}
\end{table}

\section{The set $\widehat{\rm FI}^d$}
This section aims to  classify the Dirac series for  ${\rm FI}=F_{4(4)}$, which is realized in $\texttt{atlas}$ via the command
\texttt{G:F4\_s}. This equal rank group is centerless, connected, but \emph{not} simply connected.  We adopt the simple roots of $\Delta^+(\frg, \frt_f)$ and  $\Delta^+(\frk, \frt_f)$ as in Knapp \cite[Appendix C]{Kn}.
In particular, its Vogan diagram is presented in Fig.~\ref{Fig-FI-Vogan}, where $\alpha_1, \alpha_2$ are short, while $\alpha_3, \alpha_4$ are long.
Let $\{\xi_1, \xi_2, \xi_3, \xi_4\}$ be the fundamental weights corresponding to $\Delta^+(\frg, \frt_f)$. The simple roots for $\Delta^+(\frk, \frt_f)$ are
$$
\gamma_1=\alpha_1, \quad \gamma_2=\alpha_2, \quad \gamma_3=\alpha_3, \quad \gamma_4=2\alpha_1+4\alpha_2+3\alpha_3+2\alpha_4.
$$
The Dynkin diagram for $\Delta^+(\frk, \frt_f)$ is given in Fig.~\ref{Fig-FI-K-Dynkin}.
Let $\varpi_1, \dots, \varpi_4$ be the fundamental weights for $\Delta^+(\frk, \frt_f)$.
We will express the $\texttt{atlas}$ parameters $\lambda$ and $\nu$ in terms of $\xi_1, \dots, \xi_4$, and express highest weights of $K$-types in terms of $\varpi_1, \dots, \varpi_4$. Note that the $K$-types are parameterized via the highest weight theorem by $[a, b, c, d]$ such that $a, b, c, d$ are members of $\bbN$ and that $a+c+d$ is even.

\begin{figure}[H]
\centering
\scalebox{0.6}{\includegraphics{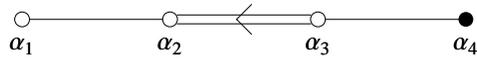}}
\caption{The Vogan diagram for FI}
\label{Fig-FI-Vogan}
\end{figure}

\begin{figure}[H]
\centering
\scalebox{0.6}{\includegraphics{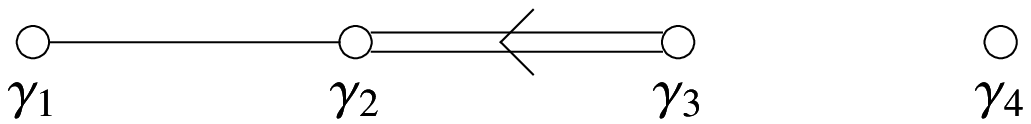}}
\caption{The Dynkin diagram for $\Delta^+(\frk, \frt_f)$}
\label{Fig-FI-K-Dynkin}
\end{figure}

Note that \texttt{atlas} labels the simple roots of $\Delta^+(\frg, \frt_f)$ in the way \textbf{opposite to} that of Fig.~\ref{Fig-FI-Vogan}. Thus whenever we put the parameters $\lambda$ and $\nu$ in the tables of this section into $\texttt{atlas}$, we should \textbf{reverse} the order of their coordinates.

%The scattered part of $\widehat{\rm FI}^{d}$ is given in Table \ref{table-FI-scattered-part}, where in the last row sits the trivial representation.

\begin{example}\label{exam-FI-scattered-part}
Let us illustrate the algorithm in Section \ref{sec-algorithm} for FI.
\begin{itemize}
\item[$\bullet$]
Step (a) gives us $369272$ candidates for $\Lambda$, among which $95849$ are not strongly regular. By Theorem \ref{thm-SR}, it suffices to focus on the latter ones.

\item[$\bullet$] FI has $229$ KGB elements in total. The following ones are fully supported:
$$
68, 69, [80, 87], 90, 93, [95, 98], 100, [103, 109], [112, 228].
$$

\item[$\bullet$] Now say fix $\#x=225$. Then only one representation survives after carrying out steps (b) and (c). This gives the nineteenth  row of Table \ref{table-FI-scattered-part}.
\end{itemize}
To sum up, the FS-scattered members of $\widehat{\rm FI}^d$ are exhausted in Table \ref{table-FI-scattered-part}, where in the last row sits the trivial representation.
\hfill\qed
\end{example}

According to the calculation carried out by Jeffrey Adams and Steve Miller in September 2019, there are eight weakly unipotent representations in Table \ref{table-FI-scattered-part}: those with $\#x=204, 215, 218, 223, 225, 228$. We note that each spin-lowest $K$-type in Table \ref{table-FI-scattered-part} is u-small, and occurs with multiplicity one.

The strings of $\widehat{\rm FI}^d$ are given in Tables
\ref{table-FI-string-part-card0}--\ref{table-FI-string-part-card3} according to $|{\rm supp}(x)|$---the cardinality of the support of $x$.
In these tables, the coordinates $a, b, c, d$ are members of $\bbN$ such that the infinitesimal character
$$
\Lambda=[d, c, b, a]
$$
and that \begin{equation}\label{FI-common-requirement}
a+b>0, \quad b+c>0, \quad c+d>0.
\end{equation}
In some cases, there are stronger requirements for certain coordinates. They will be put within the column ``spin LKTs". Every spin-lowest $K$-type in these tables occurs with multiplicity one.

\begin{table}[H]
\centering
\caption{FS-scattered members of $\widehat{\rm FI}^{\mathrm{d}}$}
\begin{tabular}{r|c|c|c|r}
$\# x$ &   $\lambda$  & $\nu$ & spin LKTs & string limit\\
\hline
$69$ & $[0,1,0,1]$ & $[-\frac{3}{2},\frac{5}{2},-\frac{3}{2},1]$ & $[3,0,1,4]$, $[4,1,0,2]$ & $\# 48$, $b=-1$\\
$83$ & $[3,1,-1,2]$ & $[3,0,-\frac{3}{2},\frac{3}{2}]$ & $[0,1,0,8]$ & \\
$106$ & $[3,-2,2,0]$ & $[3,-3,\frac{3}{2},0]$ & $[1,0,0,7]$ & $\#66$, $a=-4$\\
$108$ & $[1,2,-1,2]$ & $[0,\frac{5}{2},-\frac{5}{2},\frac{5}{2}]$ & $[6,0,0,0]$ & $\#74$, $a=-1$\\
$123$ & $[5,-4,2,3]$ & $[4,-4,\frac{3}{2},1]$ & $[2,1,0,8]$, $[3,0,0,7]$ & $\#102$, $d=-1$\\
$132$ & $[1,1,-1,4]$ & $[1,1,-2,3]$ & $[0,0,3,3]$, $[0,2,1,5]$ & 2nd $\#110$, $a=-1$\\
$133$ & $[4,-1,1,1]$  & $[\frac{9}{2},-\frac{7}{2},1,1]$ &  $[0,4,0,0]$, $[1,4,0,1]$ & 2nd $\#111$, $d=-1$\\
$142$ & $[1,3,-1,1]$ & $[0,4,-\frac{5}{2},1]$ & $[0,2,0,8]$, $[1,1,0,7]$, $[2,0,0,6]$ & $\#102$, $d=-2$\\
$147$ & $[-3,4,0,1]$ & $[-\frac{5}{2},\frac{5}{2},0,0]$ & $[4,0,1,1]$ & $\#74$, $a=-2$\\
$153$ & $[1,5,-2,1]$ &$[1,\frac{7}{2},-\frac{5}{2},1]$ & $[0,3,0,0]$, $[1,3,0,1]$, $[2,3,0,2]$& 2nd $\#111$, $d=-2$\\
$163$ & $[1,-2,4,-1]$ & $[0,-1,\frac{5}{2},-\frac{3}{2}]$ & $[0,1,0,6]$, $[1,0,1,6]$ & $\#102$, $d=-3$\\
$175$ & $[1,0,3,-2]$ & $[1,-\frac{3}{2},\frac{5}{2},-\frac{3}{2}]$ & $[1,2,1,2]$, $[2,2,0,2]$ & 2nd $\#111$, $d=-3$\\
$176$ & $[-2,3,0,1]$ & $[-2,3,-1,1]$ & $[0,0,3,1]$, $[0,2,0,4]$ & 2nd $\#110$, $a=-2$\\
$204$ & $[-2, 3,0,3]$ & $[-1,1,0,\frac{3}{2}]$ & $[0,0,1,5]$, $[0,0,1,7]$ & 2nd $\#110$, $a=-4$\\
$209$ & $[-1,4,-1,4]$ & $[-\frac{1}{2},\frac{3}{2},-\frac{1}{2},\frac{3}{2}]$& $[2,1,1,1]$, $[2,1,1,3]$ & 2nd $\#111$, $d=-4$\\
$215$ & $[3,1,0,2]$ & $[2,0,0,\frac{3}{2}]$ & $[1,0,1,4]$, $[3,0,1,6]$ &\\
$218$ & $[2,1,0,3]$ & $[\frac{3}{2},1,-\frac{1}{2},\frac{3}{2}]$ & $[0,1,2,2]$, $[0,3,0,0]$ & \\
$223$ & $[0,3,0,1]$ & $[0, 1, 0,1]$ & $[3,0,1,2]$, $[3,0,1,4]$ & 2nd $\#111$, $d=-5$\\
$225$ & $[1,3,0,1]$ & $[1,1,0,1]$ & $[0,0,1,3]$, $[2,0,1,5]$ & \\
$228$ & $[1,1,1,1]$ & $[1,1,0,1]$& $[0,0,3,3]$, $[0,2,1,1]$& \\
$228$ & $[1, 1, 1, 1]$ & $[0, 1, 0, 1]$ & $[2,0,2,0]$, $[2,0,2,2]$ & 2nd $\#111$, $d=-5$\\
$228$ & $[1, 1, 1, 1]$ & $[1, 1, 1, 1]$ & $[0, 0, 0, 0]$ &
\end{tabular}
\label{table-FI-scattered-part}
\end{table}

\begin{table}[H]
\centering
\caption{Strings of $\widehat{\rm FI}^{\mathrm{d}}$ with $|{\rm supp}(x)|=0$}
\begin{tabular}{r|c|c|c}
$\# x$ &   $\lambda$   & $\nu$ &spin LKT=LKT \\
\hline
$0$ & $[d,c,b,a]$ & $[0,0,0,0]$ & $[2b + c + 1, c+d, a+b, a+b+c + 1]$\\
$1$ & $[d,c,b,a]$ & $[0,0,0,0]$ & $[2a + 2b+c+3, c+d, b - 1, b + c]$, $b\geq1$\\
$2$ & $[d,c,b,a]$ & $[0,0,0,0]$ & $[c - 1, 2b+c+d + 2, a - 1, a+2b+c + 2]$, $a,c\geq 1$\\
$3$ & $[d,c,b,a]$ & $[0,0,0,0]$ & $[2b + c + 1, d - 1, a+b + c + 1, a+b]$, $d\geq1$\\
$4$ & $[d,c,b,a]$ & $[0,0,0,0]$ & $[2b +c +d+ 2, c - 1, a+b, a + b + c + d + 2]$, $c\geq1$\\
$5$ & $[d,c,b,a]$ & $[0,0,0,0]$ & $[2a+2b +c+ 3, d - 1, b + c, b - 1]$, $b,d\geq 1$\\
$6$ & $[d,c,b,a]$ & $[0,0,0,0]$ & $[2a + 2b + c +  d + 4, c - 1, b - 1,  b + c +d+ 1]$, $b, c\geq 1$\\
$7$ & $[d,c,b,a]$ & $[0,0,0,0]$ & $[c+d, 2b + c + 1, a- 1, a + 2b +  c + d + 3]$, $a\geq 1$\\
$8$ & $[d,c,b,a]$ & $[0,0,0,0]$ & $[c - 1, d - 1, a+2b+c + 2, a - 1]$, $a,c,d\geq 1$\\
$9$ & $[d,c,b,a]$ & $[0,0,0,0]$ & $[d - 1, 2b + c + 1, a - 1, a + 2 b + 2 c + d + 4]$, $a,d\geq 1$\\
$10$ & $[d,c,b,a]$ & $[0,0,0,0]$ & $[d - 1, c - 1, a+b, a + 3 b + 2 c + d + 5]$, $c,d\geq 1$\\
$11$ & $[d,c,b,a]$ & $[0,0,0,0]$ & $[d - 1, c - 1, b - 1, 2a + 3b +2c +d + 6]$, $b,c,d\geq 1$
\end{tabular}
\label{table-FI-string-part-card0}
\end{table}

\begin{cor}
Let $G$ be {\rm FI}. Then all the $K$-types whose spin norm is equal to their lambda norm are exactly the ones in the last column of Table \ref{table-FI-string-part-card0}.
\end{cor}
\begin{proof}
Note that the strings with $0\leq \#x\leq 11$ give precisely all the irreducible tempered representations of FI with non-zero Dirac cohomology. The result follows from Theorem 1.2 of \cite{DD16}.
\end{proof}

\begin{example}\label{exam-counter}
Let us look at the Dirac cohomology of the scattered representation $\pi$ with $\#x=176$ carefully. As noted in Table \ref{table-FI-scattered-part}, $\pi$ has two spin lowest $K$-types:
$$
[0,0,3,1], \quad [0,2,0,4].
$$
We enumerate the set $W(\frg, \frt_f)^1$ as follows:
\begin{align*}
w^{(0)}&=e, \quad w^{(1)}=s_4, \quad w^{(2)}=s_4s_3, \quad w^{(3)}=s_4s_3s_2, \quad w^{(4)}=s_4s_3s_2s_1,  \\
w^{(5)}&=s_4s_3s_2s_3, \quad w^{(6)}=s_4s_3s_2s_1s_3, \quad w^{(7)}=s_4s_3s_2s_3s_4, \quad w^{(8)}=s_4s_3s_2s_1s_3s_2, \\
w^{(9)}&=s_4s_3s_2s_1s_3s_4,\quad w^{(10)}=s_4s_3s_2s_1s_3s_2s_3, \quad w^{(11)}=s_4s_3s_2s_1s_3s_2s_4.
\end{align*}
For $0\leq j\leq 11$, put $\rho_n^{(j)}=w^{(j)}\rho-\rho_c$. They are listed as follows:
\begin{align*}
\rho_n^{(0)}&=[0, 0, 0, 7], \quad \rho_n^{(1)}=[0, 0, 1, 6], \quad \rho_n^{(2)}=[0, 2, 0, 5], \quad \rho_n^{(3)}=[1, 2, 0, 4], \\
\rho_n^{(4)}&=[0, 3, 0, 3], \quad \rho_n^{(5)}=[3, 0, 1, 3], \quad \rho_n^{(6)}=[2, 1, 1, 2], \quad \rho_n^{(7)}=[5, 0, 0, 2],\\
\rho_n^{(8)}&=[2, 0, 2, 1], \quad \rho_n^{(9)}=[4, 1, 0, 1], \quad \rho_n^{(10)}=[0, 0, 3, 0], \quad \rho_n^{(11)}=[4, 0, 1, 0].
\end{align*}
Since the longest element of $W(\frk, \frt_f)$ is $-1$, the lowest weight of the $\frk$-type with highest weight $\rho_n^{(j)}$ is always $-\rho_n^{(j)}$.
One computes that the spin lowest $K$-type $[0,2,0,4]$ contributes a unique $\widetilde{K}$-type to $H_D(\pi)$:
$$
\{[0,2,0,4]-\rho_n^{(2)}\}=\{[0,2,0,4]-[0,2,0,5]\}=[0,0,0,1].
$$
Since $w^{(2)}=s_4s_3$ has even length, this $\widetilde{K}$-type actually lives in $H_D^+(\pi)$, the even part of $H_D(\pi)$.

On the other hand, the spin lowest $K$-type $[0,0,3,1]$ contributes a unique $\widetilde{K}$-type to $H_D(\pi)$:
$$
\{[0,0,3,1]-\rho_n^{(10)}\}=\{[0,0,3,1]-[0,0,3,0]\}=[0,0,0,1].
$$
Since $w^{(10)}=s_4s_3s_2s_1s_3s_2s_3$ has odd length, this $\widetilde{K}$-type actually lives in $H_D^-(\pi)$, the odd part of $H_D(\pi)$.

This gives a counter-example to Conjecture 10.3 of \cite{H15} asserting that
$$
{\rm Hom}_{\widetilde{K}}(H_D^+(X), H_D^-(X))=0
$$
for any irreducible $(\frg, K)$-module $X$ whenever $G$ is equal rank.

Note that $H_D(\pi)\neq 0$, while the Dirac index $H_D^+(\pi)-H_D^-(\pi)$ vanishes.
\hfill\qed
\end{example}
\begin{rmk} (a)\
In view of Theorem 8.3 of \cite{H15}, $\pi$ is \emph{not} an elliptic representation since its Dirac index is zero. Thus $\pi$ may also violate Conjecture 13.2 of \cite{H15}.

\smallskip

\noindent (b)
One can realize $\pi$ as a string limit according to the last column of Table \ref{table-FI-scattered-part}. In this way one checks that $\pi$ is actually a fair $A_{\mathfrak{q}}(\lambda)$-module.

\smallskip

\noindent (c)\ By using the translation principle \cite{MPV}, another way of showing that the Dirac index of $\pi$ vanishes has been given in \cite{DW}.
\end{rmk}

\begin{table}[H]
\centering
\caption{Strings of $\widehat{\rm FI}^{\mathrm{d}}$ with $|{\rm supp}(x)|=1$}
\begin{tabular}{r|c|c|c|r}
$\# x$ &   $\lambda$   & $\nu$ &spin LKTs  \\
\hline
$12$ & $[1, c, b, a]$ & $[1,-\frac{1}{2},0,0]$ & $[2b+c+2, c, a+b, a+b+c+2]$\\
$13$ & $[1, c, b, a]$ & $[1,-\frac{1}{2},0,0]$ & $[2a+2b+c+4, c, b-1, b+c+1]$, $b\geq 1$\\
$14$ & $[1, c, b, a]$ & $[1,-\frac{1}{2},0,0]$ & $[c, 2b+c+2, a-1, a+2b+c+3]$, $a\geq 1$\\
$15$ & $[d, 1, b, a]$ & $[-\frac{1}{2},1,-\frac{1}{2},0]$ & $[2b+2, d, a+b+1, a+b+1]$\\
$16$ & $[d, 1, b, a]$ & $[-\frac{1}{2},1,-\frac{1}{2},0]$ & $[2a+2b+4, d, b, b]$, $b\geq 1$\\
$17$ & $[d, 1, b, a]$ & $[-\frac{1}{2},1,-\frac{1}{2},0]$ & $[d, 2b+2, a-1, a+2b+d+5]$, $a\geq 1$\\
$18$ & $[d, c, 1, a]$ &  $[0,-1,1,-\frac{1}{2}]$ & $[c+1, c+d+2, a, a+c+3]$\\
$19$ & $[d, c, 1, a]$ &  $[0,-1,1,-\frac{1}{2}]$ & $[c+d+2, c+1, a, a+c+d+4]$, $a+c\geq 1$\\
$20$ & $[d, c, 1, a]$ &  $[0,-1,1,-\frac{1}{2}]$ & $[c+1, d-1, a+c+3, a]$, $d\geq 1$\\
$21$ & $[d, c, 1, a]$ &  $[0,-1,1,-\frac{1}{2}]$ & $[d-1, c+1, a, a+2c+d+7]$, $a+c\geq 1$, $d\geq 1$\\
$22$ & $[d, c, b, 1]$ & $[0, 0, -\frac{1}{2}, 1]$ & $[2b+c+3, c+d, b, b+c+1]$\\
$23$ & $[d, c, b, 1]$ & $[0, 0, -\frac{1}{2}, 1]$ & $[2b+c+3, d-1, b+c+1, b]$, $d\geq 1$\\
$24$ & $[d, c, b, 1]$ & $[0, 0, -\frac{1}{2}, 1]$ & $[2b+c+d+4, c-1, b, b+c+d+2]$, $c\geq 1$\\
$25$ & $[d, c, b, 1]$ & $[0, 0, -\frac{1}{2}, 1]$ & $[d-1, c-1, b, 3b+2c+d+7]$, $c\geq 1, d\geq 1$
\end{tabular}
\label{table-FI-string-part-card1}
\end{table}

\begin{table}[H]
\centering
\caption{Strings of $\widehat{\rm FI}^{\mathrm{d}}$ with $|{\rm supp}(x)|=2$}
\begin{tabular}{r|c|c|c|r}
$\# x$ &   $\lambda$   & $\nu$ &spin LKTs  \\
\hline
$26$ & $[1,c,1,a]$ & $[1,-\frac{3}{2},1,-\frac{1}{2}]$ & $[c+2,c+2,a, a+c+4]$\\
$27$ & $[1,c,b,1]$ & $[1, -\frac{1}{2},-\frac{1}{2},1]$& $[2b+c+4,c,b,b+c+2]$\\
$28$ & $[d,1,b,1]$ & $[-\frac{1}{2}, 1,-1,1]$ & $[2b+4, d,b+1,b+1]$\\
$29$ & $[1,1,b,a]$ & $[1,1,-1,0]$ & $[2b+3,0,a+b+1, a+b+2]$\\
$30$ & $[1,1,b,a]$ & $[1,1,-1,0]$ & $[2a+2b+5,0,b,b+1]$, $b\geq 1$\\
$31$ & $[1,1,b,a]$ & $[1,1,-1,0]$ & $[0,2b+3,a-1,a+2b+5]$, $a\geq 1$\\
$32$ & $[d-2,3,0,a-1]$ & $[-1,1,0,-\frac{1}{2}]$ & $[0,d+2,a,a+2]$, $a\geq 1$, $d\geq 1$\\
$33$ & $[d-2,3,0,a-1]$ & $[-1,1,0,-\frac{1}{2}]$ & $[0,d,a+2,a]$, $a\geq 1$\\
$34$ & $[d-2,3,0,a-1]$ & $[-1,1,0,-\frac{1}{2}]$ & $[d+2,0,a,a+d+4]$, $a\geq 1$, $d\geq 1$\\
$35$ & $[d-2,3,0,a-1]$ & $[-1,1,0,-\frac{1}{2}]$ & $[d,0,a,a+d+6]$, $a\geq 1$\\
$37$ & $[d, 1,1,a]$ & $[-\frac{3}{2},0, \frac{3}{2},-\frac{3}{2}]$ & $[2a+6,d+1,0,0]$\\

$39$ & $[d,c,1,1]$ & $[0,-2,1,1]$ & $[c+3,c+d+2,0,c+3]$, $c+d\geq 1$\\
$40$ & $[d,c,1,1]$ & $[0,-2,1,1]$ & $[c+d+4,c+1,0,c+d+4]$, $c+d\geq 1$\\
$41$ & $[d,c,1,1]$ & $[0,-2,1,1]$ & $[c+3,d-1,c+3,0]$, $d\geq 1$\\
$42$ & $[d,c,1,1]$ & $[0,-2,1,1]$ & $[d-1,c+1,0,2c+d+9]$, $d\geq 1$\\
\hline
$44$ & $[d-1,1,1,a-1]$ & $[-1,1,0,-\frac{1}{2}]$& $[2, d+1, a, a+2]$, $a\geq 1$, $d\geq 0$\\
& & & $[2,d-1,a+2,a]$, $a\geq 1$, $d\geq 1$\\
\hline
$44$ & $[d,1,1,a]$ & $[-2,1,1,-\frac{3}{2}]$& $[0, d+2, a+2, a+2]$\\
\hline
$45$ & $[d-1,1,1,a-1]$ & $[-1,1,0,-\frac{1}{2}]$& $[d+1,2, a-1, a+d+4]$, $a\geq 1$, $d\geq 0$ \\
& & & $[d-1,2, a-1,a+d+6]$, $a\geq 1$, $d\geq 1$\\
\hline
$45$ & $[d,1,1,a]$ & $[-2,1,1,-\frac{3}{2}]$& $[d+2,0, a+1, a+d+7]$\\
\end{tabular}
\label{table-FI-string-part-card2}
\end{table}

\begin{table}[H]
\centering
\caption{Strings of $\widehat{\rm FI}^{\mathrm{d}}$ with $|{\rm supp}(x)|=3$}
\begin{tabular}{r|c|c|c|r}
$\# x$ &   $\lambda$   & $\nu$ &spin LKTs  \\
\hline
$43$ & $[0,1,0,a]$ & $[-\frac{1}{2},\frac{3}{2},-\frac{1}{2},-\frac{1}{2}]$ & $[1,2,a-1,a+4]$, $a\geq 1$; $[2,1,a,a+2]$, $a\geq 1$\\
$50$ & $[3,1,0,a-1]$ & $[2,0,0,-1]$ & $[1,0,a+2,a+1]$, $a\geq 1$\\
$52$ & $[3,1,0,a-1]$ & $[2,0,0,-1]$ & $[0,1,a,a+6]$, $a\geq 1$\\
$54$ & $[2,-1,2,a-1]$ & $[\frac{3}{2},-\frac{3}{2},\frac{3}{2},-\frac{3}{2}]$ & $[2a+5,1,0,1]$, $a\geq 1$\\
$65$ & $[5,-2,3,a-3]$ & $[2,-1,1,-\frac{3}{2}]$ & $[1,1, a+2, a+1]$, $a\geq 0$; $[0,1,a+2,a+2]$, $a\geq 1$\\
$66$ & $[5,-2,3,a-3]$ & $[2,-1,1,-\frac{3}{2}]$ & $[1,0,a+1,a+6]$, $a\geq 0$; $[1,1,a,a+7]$, $a\geq 1$\\
$73$ & $[0,3,0,a-2]$ & $[0,1,0,-1]$ & $[3,0,a,a+3]$, $a\geq 1$\\
$74$ & $[1,1,1,a]$ & $[0,\frac{5}{2},0,-\frac{5}{2}]$ & $[2a+8,0,0,0]$\\
$75$ & $[0,3,0,a-2]$ & $[0,1,0,-1]$ & $[0,3,a-1,a+3]$, $a\geq 1$\\
$88$ & $[1, 2, 0, a-1]$ &  $[1,1,0,-\frac{3}{2}]$  & $[0,3,a,a+2]$, $a\geq 1$; $[2,1,a,a+4]$, $a\geq 1$\\
$88$ & $[1, 3, 0, a-2]$ &  $[1,1,0,-\frac{3}{2}]$  & $[0,0,a+2,a+2]$, $a\geq 1$; $[2,0,a+2,a]$, $a\geq 1$\\
$89$ & $[1, 2, 0, a-1]$ &  $[1,1,0,-\frac{3}{2}]$  & $[1,2,a,a+3]$, $a\geq 1$; $[3,0,a,a+5]$, $a\geq 1$\\
$89$ & $[1, 3, 0, a-2]$ &  $[1,1,0,-\frac{3}{2}]$  & $[0,0,a+1,a+5]$, $a\geq 1$; $[0,2,a-1,a+7]$, $a\geq 1$\\
$110$ & $[1, 1, 1, a-2]$ &  $[0,1,0,-1]$  & $[0,2,a-1,a+5]$, $a\geq 1$; $[2,0, a+1,a+1]$, $a\geq 1$\\
$110$ & $[1, 1, 1, a]$ &  $[1,1,1,-3]$  & $[0,0,a+3,a+5]$\\
$48$ & $[1, b, 1, 1]$ &  $[1,-\frac{5}{2},1,1]$  & $[b+4,b+2,0,b+4]$\\
$47$ & $[1, 1, c, 1]$ &  $[1,1,-\frac{3}{2},1]$  & $[2c+5,0,c+1,c+2]$\\
$56$ & $[d-2,3,-1,2]$ & $[-2,2,-1,1]$ & $[2, d+1,1,3]$, $d\geq 0$;
$[2,d+2,0,2]$, $d\geq 1$\\
$57$ & $[d-2,3,-1,2]$ & $[-2,2,-1,1]$ & $[d+4,0,0,d+4]$, $d\geq 1$ \\
$58$ & $[d-2,3,-1,2]$ & $[-2,2,-1,1]$ & $[2,d,2,0], d\geq 0$; $[2,d-1,3,1], d\geq 1$\\
$59$ & $[d-2,3,-1,2]$& $[-2,2,-1,1]$ & $[d,0,0,d+8]$ \\
$61$ & $[d-1,1,0,2]$ & $[-\frac{3}{2},0,0,\frac{3}{2}]$ & $[4,d+1,0,0]$ \\
$70$ & $[d-2,1,0,3]$& $[-2,1,-\frac{1}{2},\frac{3}{2}]$ & $[2,d,2,2]$ \\
$71$ & $[d-1,1,0,2]$ & $[-2,1,-\frac{1}{2},\frac{3}{2}]$ & $[d+3,0,1,d+4]$, $d\geq 0$; $[d-1,0,1,d+8]$, $d\geq 1$ \\
$71$ & $[d-2,1,0,3]$ & $[-2,1,-\frac{1}{2},\frac{3}{2}]$& $[d,0,1,d+7]$, $d\geq 0$; $[d+2,0,1, d+5]$, $d\geq 1$ \\
$76$ & $[d-2,0,2,0]$ & $[-2,0,1,0]$ & $[1,d+1,1,2]$, $d\geq 1$; $[3, d,1,2]$, $d\geq 1$\\
$77$ & $[d-2,0,2,0]$ & $[-2,0,1,0]$ & $[1,d,2,1]$, $d\geq 1$; $[3, d-1,2,1]$, $d\geq 1$\\
$99$ & $[d-4,3,0,2]$ & $[-\frac{5}{2},1,0,1]$ & $[4,d,1,1]$\\
$102$ & $[d,1,1,1]$ & $[-4,0,\frac{3}{2},1]$ & $[d+4,0,0,d+8]$\\
$111$ & $[d-2,1,1,1]$ & $[-\frac{5}{2},1,0,1]$ & $[2,d+1,1,1]$\\
$111$ & $[d,1,1,1]$ & $[-\frac{9}{2},1,1,1]$ & $[0,d+5,0,0]$
\end{tabular}
\label{table-FI-string-part-card3}
\end{table}

Finally, let us present an example saying that a spin-lowest $K$-type of an irreducible unitary representation could have multiplicity bigger than one.

\begin{example}\label{exam-spin-mult-big}
Let us consider the irreducible representation of FI with parameter $p=(x, \lambda, \nu)$, where $\#x=81$ which is fully supported, $\lambda=[4,0,-1,2]$ and $\nu=[2,0,-1,1]$.  This representation is unitary, and has infinitesimal character $\Lambda=[1,0,1/2, 1/2]$. (Recall that we should reverse the coordinates of $\lambda$, $\nu$  and $\Lambda$ for \texttt{atlas}.) It has a unique lambda-lowest $K$-type $[2,1,0,4]$, and has five spin-lowest $K$-types
$$
[0, 2, 0, 6], \quad [1, 2, 0, 5], \quad [3, 0, 1, 4], \quad [0, 3, 0, 4], \quad [2, 1, 1, 3].
$$
Their multiplicities are $1, 2, 1, 1, 1$, respectively. Note that this representation has zero Dirac cohomology. Indeed, it has spin norm $3$, which is strictly larger than $\|\Lambda\|=\sqrt{\frac{15}{2}}$.
\hfill\qed
\end{example}

\section{The set $\widehat{\rm EI}^d$}
This section aims to classify the Dirac series for ${\rm EI}=E_{6(6)}$, which
 is realized in $\texttt{atlas}$ via the command
\texttt{G:E6\_s}. This group is centerless, connected, but \emph{not} simply connected.
It is not equal rank. Indeed, $\dim \frt_f=4$ and $\dim a_f=2$.

 We adopt the simple roots of $\Delta^+(\frg, \frt_f)$ and  $\Delta^+(\frk, \frt_f)$ as in Knapp \cite[Appendix C]{Kn}.  In particular, its Vogan diagram is presented in Fig.~\ref{Fig-EI-Vogan}. Let $\{\xi_1, \dots, \xi_6\}$ be the fundamental weights of $\Delta^+(\frg, \frh_f)$. The simple roots for $\Delta^+(\frg, \frt_f)$ are
$$
\alpha_4:=\beta_2, \quad \alpha_3:=\beta_4, \quad \alpha_2:=\frac{1}{2}(\beta_3+\beta_5), \quad \alpha_1:=\frac{1}{2}(\beta_1+\beta_6).
$$
The root system $\Delta^+(\frg, \frt_f)$ is $F_4$, with $\alpha_1, \alpha_2$ short and $\alpha_3, \alpha_4$ long. On the other hand, $\Delta^+(\frk, \frt_f)$ is $C_4$, and has simple roots
$$
\gamma_1:=\alpha_2+\alpha_3+\alpha_4, \quad \gamma_2:=\alpha_1, \quad \gamma_3:=\alpha_2, \quad \gamma_4:=\alpha_3.
$$
Here $\gamma_4$ is long.
Accordingly, let $\varpi_1, \dots, \varpi_4$ be the fundamental weights for $\Delta^+(\frk, \frt_f)$.
We will express the $\texttt{atlas}$ parameters $\lambda$ and $\nu$ in terms of $\xi_1, \dots, \xi_6$, and express highest weights of $K$-types in terms of $\varpi_1, \dots, \varpi_4$. Note that the $K$-types are parameterized via the highest weight theorem by $[a, b, c, d]$ such that $a, b, c, d$ are members of $\bbN$ and that $a+c$ is even.

\begin{figure}[H]
\centering
\scalebox{0.6}{\includegraphics{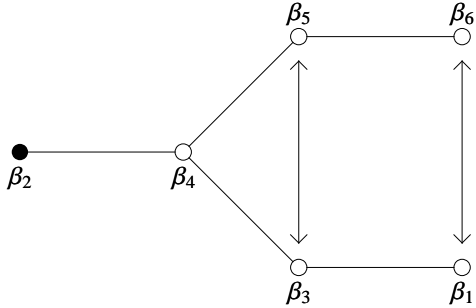}}
\caption{The Vogan diagram for EI}
\label{Fig-EI-Vogan}
\end{figure}

\begin{example}\label{exam-EI-Dirac-inequality}

Like in the case of complex $E_6$ \cite{D17E6}, distribution of the spin norm along Vogan pencils turns out to be very effective for detecting non-unitarity for  ${\rm EI}=E_{6(6)}$. Let us consider $\Lambda=[1,5,5,0,5,1]$. Again, to save space, certain outputs of \texttt{atlas} have been omitted.
\begin{verbatim}
G:E6_s
set all=all_parameters_gamma(G, [1,5,5,0,5,1])
#all
Value: 1258
\end{verbatim}
The last output says that there are $1258$ irreducible representations with infinitesimal character $[1,5,5,0,5,1]$ in total. We check that each representation is infinite-dimensional via the following command.
\begin{verbatim}
for p in all do if is_finite_dimensional(p) then prints(p) fi od
\end{verbatim}
Now we can apply Theorem C of \cite{D17} to these representations.
The following command prints the highest weight of one of the lowest $K$-types for each representation.
\begin{verbatim}
for p in all do prints(highest_weight(LKTs(p)[0], KGB(G,0))) od
\end{verbatim}
By calculating the minimum spin norm along the Vogan pencils starting from the lowest $K$-types above, and using Parthasarathy's Dirac operator inequality, we find that $1254$ of them must be non-unitary. Then we check via the command \texttt{is\_unitary} that the remaining four representations are all unitary:
\begin{verbatim}
final parameter(x=10,lambda=[1,5,5,0,5,1]/1,nu=[2,0,-1,0,-1,2]/2)
final parameter(x=8,lambda=[1,5,5,0,5,1]/1,nu=[2,0,-1,0,-1,2]/2)
final parameter(x=2,lambda=[1,5,5,0,5,1]/1,nu=[0,0,0,0,0,0]/1)
final parameter(x=0,lambda=[1,5,5,0,5,1]/1,nu=[0,0,0,0,0,0]/1)
\end{verbatim}
 All of them turn out to have non-zero Dirac cohomology. Indeed, their spin-lowest $K$-types are $[11,0,13,4]$, $[13,0,11,5]$, $[10,2,12,4]$, $[12,2,10,5]$, respectively; moreover, they all have spin norm $\sqrt{634}$, which equals $\|\Lambda\|$. Now the reader sees that the four representations merge into four strings in Tables \ref{table-EI-string-part-card0} and \ref{table-EI-string-part-card2}.
\hfill\qed
\end{example}

Carrying out the algorithm in Section \ref{sec-algorithm} for EI leads us to  Table \ref{table-EI-scattered-part}, where in the last row sits the trivial representation.
We note that each spin-lowest $K$-type in Table \ref{table-EI-scattered-part} is u-small, and occurs with multiplicity one.

The strings of $\widehat{\rm EI}^d$ are given in Tables
\ref{table-EI-string-part-card0}--\ref{table-EI-string-part-card5} according to $|{\rm supp}(x)|$---the cardinality of the support of $x$.
In these tables, the coordinates $a, b, c, d, e, f$ are members of $\bbN$ such that the infinitesimal character
$$
\Lambda=[\frac{a+f}{2}, b, \frac{c+e}{2}, d, \frac{c+e}{2}, \frac{a+f}{2}]
$$
and that
\begin{equation}\label{EI-common-requirement}
a-f=0 \mbox{ or } 1, \quad c- e=0 \mbox{ or } 1, \quad a+f\geq 1, \quad c+e\geq 1, \quad b+d\geq 1.
\end{equation}
In some cases, there are stronger requirements for certain coordinates. They will be put within the column ``spin LKTs". Every spin-lowest $K$-type in these tables occurs with multiplicity one.

\begin{cor}
Let $G$ be {\rm EI}. Then all the $K$-types whose spin norm is equal to their lambda norm are exactly the ones in the last column of Table \ref{table-EI-string-part-card0}.
\end{cor}
\begin{proof}
Note that the strings in Table \ref{table-EI-string-part-card0} are precisely all the irreducible tempered representations of EI with non-zero Dirac cohomology. The result follows from Theorem 1.2 of \cite{DD16}.
\end{proof}

We note that different parameters may represent the same module in \texttt{atlas}.
We will always choose one way to represent certain strings uniformly.
\begin{example}
Let us look at the string with $\#x=3$ in Table \ref{table-EI-string-part-card1} more closely. Take $a=b=c=1$, $e=f=0$. Then the representation is \texttt{p} below. However, $\texttt{atlas}$ will change it into another form, namely \texttt{pp} below.
\begin{verbatim}
set p=parameter(KGB(G,3), [1,1,1,1,0,0],[0,-1,-1,2,-1,0]/2)
p
Value: final parameter(x=3,lambda=[0,0,-1,3,0,1]/1,nu=[0,-1,-1,2,-1,0]/2)
set pp=parameter(KGB(G,3), [0,0,-1,3,0,1],[0,-1,-1,2,-1,0]/2)
pp=p
Value: true
\end{verbatim}
The last output confirms that both \texttt{p} and \texttt{pp} stand for the same representation.\hfill\qed
\end{example}

\begin{table}[H]
\centering
\caption{FS-scattered members of $\widehat{\rm EI}^{\mathrm{d}}$}
\begin{tabular}{r|c|c|c|r}
$\# x$ &   $\lambda$  & $\nu$ & spin LKTs & string limit \\
\hline
$109$ & $[0,2,2,-2,2,0]$ & $[-\frac{1}{2},1,\frac{3}{2},-2,\frac{3}{2},-\frac{1}{2}]$ & $ [5,1,1,0]$, $[3,1,1,1]$ & $\#37$, $c+e=-1$ \\
$137$ & $[1,2,0,0,0,1]$ & $[1,2,-1,0,-1,1]$ & $[5,1,1,0]$ & $\#43$, $a+f=-1$\\

$209$ & $[3,4,0,-1,-1,2]$ & $[\frac{3}{2},2,-\frac{1}{2},-1,-\frac{1}{2},\frac{3}{2}]$ & $[3,1,1,1]$ & $\#90$, $a+f=-1$\\

$270$ & $[1,1,1,0,1,1]$ & $[0,4,2,-4,2,0]$ & $[10,0,0,0]$ & $\#193$, $b=-1$ \\

$338$ & $[4,1,-1,0,0,3]$ & $[2,1,-1,0,-1,2]$ & $[5,1,1,0]$, $[3,1,1,1]$ & \\
$373$ & $[2,1,0,1,0,2]$ & $[3,1,-2,1,-2,3]$ & $[1,5,1,0]$ & $\#204$, $a+f=-1$\\
$448$ & $[3,9,1,-4,1,3]$ & $[1,\frac{9}{2},1,-\frac{7}{2},1,1]$ & $[0,0,2,3]$ & 3rd $\#359$, $b=-1$ \\
$567$ & $[0,3,5,-5,5,0]$ & $[1,1,2,-3,2,1]$ & $[2,2,2,0]$ & $\#204$, $a+f=-3$\\
$863$ & $[0,3,3,-2,2,1]$ & $[0,2,1,-1,1,0]$ & $[1,1,3,0]$, $[3,1,1,1]$ & $\#204$, $a+f=-5$\\
$981$ & $[2,1,2,1,2,2]$ & $[\frac{1}{2},1,\frac{1}{2},0,\frac{1}{2},\frac{1}{2}]$ & $[1,1,3,0]$, $[3,1,1,1]$ & $\#204$, $a+f=-7$\\
$981$ & $[1,1,1,1,1,1]$ & $[\frac{1}{2},1,\frac{1}{2},0,\frac{1}{2},\frac{1}{2}]$ & $[1,1,3,0]$, $[3,1,1,1]$ & $\#204$, $a+f=-7$\\
$981$ & $[1,1,1,1,1,1]$ & $[1,1,1,0,1,1]$ & $[0,0,0,3]$ & \\
$981$ & $[1,1,1,1,1,1]$ & $[1,1,1,1,1,1]$ & $[0,0,0,0]$ &
\end{tabular}
\label{table-EI-scattered-part}
\end{table}

\begin{table}[H]
\centering
\caption{Strings of $\widehat{\rm EI}^{\mathrm{d}}$ with $|{\rm supp}(x)|=0$}
\begin{tabular}{r|c|c|c|r}
$\# x$ &   $\lambda$   & $\nu$ &spin LKT=LKT  \\
\hline
$0$ & $[a,b,c,d,e,f]$ & $[0,0,0,0,0,0]$ & $[c+2d+e+2,a+f,c+e,b+d]$\\
$1$ & $[a,b,c,d,e,f]$ & $[0,0,0,0,0,0]$ & $[2b+c+2d+e+4,a+f,c+e,d-1]$, $d\geq 1$\\
$2$ & $[a,b,c,d,e,f]$ & $[0,0,0,0,0,0]$ & $[c+e,a+f,c+2d+e+2,b-1]$, $b\geq 1$
\end{tabular}
\label{table-EI-string-part-card0}
\end{table}

\begin{table}[H]
\centering
\caption{Strings of  $\widehat{\rm EI}^{\mathrm{d}}$ with $|{\rm supp}(x)|=1$}
\begin{tabular}{r|c|c|c|r}
$\# x$ &   $\lambda$   & $\nu$ &spin LKTs  \\
\hline
$3$ & $[a, b, c, 1, e, f]$ & $[0, -\frac{1}{2}, -\frac{1}{2}, 1, -\frac{1}{2}, 0 ]$ & $[c+e+2, a+f, c+e+2, b]$\\
$4$ & $[a, 1, c, d, e, f]$ & $[0, 1, 0, -\frac{1}{2}, 0, 0]$ & $[c+2d+e+4, a+f, c+e, d]$\\
\end{tabular}
\label{table-EI-string-part-card1}
\end{table}

\begin{table}[H]
\centering
\caption{Strings of  $\widehat{\rm EI}^{\mathrm{d}}$ with $|{\rm supp}(x)|=2$}
\begin{tabular}{r|c|c|c|r}
$\# x$ &   $\lambda$   & $\nu$ &spin LKTs  \\
\hline
$5$ & $[a,b,1,d,1,f]$ & $[-\frac{1}{2},0,1,-1,1, -\frac{1}{2}]$ & $[2d+4,a+f+1,0,b+d+1]$\\
$6$ & $[a,b,1,d,1,f]$ & $[-\frac{1}{2},0,1,-1,1, -\frac{1}{2}]$ & $[2b+2d+6,a+f+1,0,d]$, $d\geq 1$\\
$7$ & $[a,b,1,d,1,f]$ & $[-\frac{1}{2},0,1,-1,1, -\frac{1}{2}]$ & $[0,a+f+1,2d+4,b-1]$, $b\geq 1$\\
$8$ & $[1,b,c,d,e,1]$ & $[1,0,-\frac{1}{2},0,-\frac{1}{2},1]$ & $[c+2d+e+3,0,c+e+1,b+d]$\\
$9$ & $[1,b,c,d,e,1]$ & $[1,0,-\frac{1}{2},0,-\frac{1}{2},1]$ & $[2b+c+2d+e+5,0,c+e+1,d-1]$, $d\geq 1$\\
$10$ & $[1,b,c,d,e,1]$ & $[1,0,-\frac{1}{2},0,-\frac{1}{2},1]$ & $[c+e+1,0,c+2d+e+3,b-1]$, $b\geq 1$\\
$13$ & $[a,1,c,1,e,f]$ & $[0,1,-1,1,-1,0]$ & $[c+e+4,a+f,c+e+2,0]$
\end{tabular}
\label{table-EI-string-part-card2}
\end{table}

\begin{table}[H]
\centering
\caption{Strings of  $\widehat{\rm EI}^{\mathrm{d}}$ with $|{\rm supp}(x)|=3$}
\begin{tabular}{r|c|c|c|r}
$\# x$ &   $\lambda$   & $\nu$ &spin LKTs  \\
\hline
$14$ & $[a,1,1,d,1,f]$ & $[-\frac{1}{2},1,1,-\frac{3}{2},1,-\frac{1}{2}]$ & $[2d+6,a+f+1,0,d+1]$\\
$15$ & $[1,b,c,1,e,1]$ & $[1,-\frac{1}{2},-1,1,-1,1]$ & $[c+e+3,0,c+e+3,b]$\\
$16$ & $[1,1,c,d,e,1]$ & $[1,1,-\frac{1}{2},-\frac{1}{2},-\frac{1}{2},1]$ & $[c+2d+e+5,0,c+e+1,d]$\\
$18$ & $[a,b,1,1,1,f]$ & $[-1,-2,0,2,0,-1]$ & $[2b+8,a+f+2,0,0]$\\
$29$ & $[a-3,b-2,3,0,2,f-1]$ & $[-\frac{1}{2},-\frac{1}{2},\frac{1}{2},0,\frac{1}{2},-\frac{1}{2}]$ & $[1,a+f,3,b-1]$, $[3,a+f,1,b]$
\end{tabular}
\label{table-EI-string-part-card3}
\end{table}

\begin{table}[H]
\centering
\caption{Strings of  $\widehat{\rm EI}^{\mathrm{d}}$ with $|{\rm supp}(x)|=4$}
\begin{tabular}{r|c|c|c|r}
$\# x$ &   $\lambda$   & $\nu$ &spin LKTs  \\
\hline
$37$ & $[1,1,c,1,e,1]$ & $[1,1,-\frac{3}{2},1,-\frac{3}{2},1]$ & $[c+e+5,0,c+e+3,0]$\\
$43$ & $[a-1,2,1,0,1,f]$ & $[-1,2,0,0,0,-1]$ & $[6,a+f+2,0,0]$\\
$48$ & $[1,b,1,d,1,1]$ & $[1,0,1,-2,1,1]$ & $[2d+6,0,0,b+d+2]$\\
$48$ & $[1,b,1,d-1,1,1]$ & $[\frac{1}{2},0,\frac{1}{2},-1,\frac{1}{2},\frac{1}{2}]$ & $[2d+3,1,1,b+d]$\\
$49$ & $[1,b,1,d,1,1]$ & $[1,0,1,-2,1,1]$ & $[2b+2d+8,0,0,d+1]$, $d\geq 1$\\
$49$ & $[1,b,1,d-1,1,1]$ & $[\frac{1}{2},0,\frac{1}{2},-1,\frac{1}{2},\frac{1}{2}]$ & $[2b+2d+5,1,1,d-1]$, $d\geq 1$\\

$50$ & $[1,b,1,d,1,1]$ & $[1,0,1,-2,1,1]$ & $[0,0,2d+6,b-1]$, $b\geq 1$\\
$50$ & $[1,b,1,d-1,1,1]$ & $[\frac{1}{2},0,\frac{1}{2},-1,\frac{1}{2},\frac{1}{2}]$ & $[1,1,2d+3,b-1]$, $b\geq 1$\\
$51$ & $[a-1,b-1,1,1,1,f]$ & $[-1,-1,1,0,1,-1]$ & $[2,a+f+1,2,b]$, $b\geq 1$\\
$51$ & $[a,b,1,1,1,f]$ & $[-\frac{3}{2},-2,1,1,1,-\frac{3}{2}]$ & $[0,a+f+4,0,b+2]$\\
$62$ & $[a-3,3,3,-2,2,f-1]$ & $[-1,1,1,-1,1,-1]$ & $[3,a+f,1,1]$, $[5,a+f,1,0]$\\
$90$ & $[a,1,1,0,1,f]$ & $[-\frac{3}{2},2,1,-1,1,-\frac{3}{2}]$ & $[2,a+f+2,0,2]$\\
$162$ & $[a-2,1,2,0,2,f-1]$ & $[-2,1,1,0,1,-2]$ & $[4,a+f+2,0,1]$\\
$204$ & $[a,1,1,1,1,f]$ & $[-3,1,1,1,1,-3]$ & $[0,a+f+7,0,0]$
\end{tabular}
\label{table-EI-string-part-card4}
\end{table}

\begin{table}[H]
\centering
\caption{Strings of  $\widehat{\rm EI}^{\mathrm{d}}$ with $|{\rm supp}(x)|=5$}
\begin{tabular}{r|c|c|c|r}
$\# x$ &   $\lambda$   & $\nu$ &spin LKTs  \\
\hline
$58$ & $[0,b-1,2,-1,2,0]$ & $[0,-\frac{1}{2},1,-1,1,0]$ & $[1,1,3,b-1]$, $[3,1,1,b]$\\
$73$ & $[1,1,1,d,1,1]$ & $[1,1,1,-\frac{5}{2},1,1]$ & $[2d+8,0,0,d+2]$\\
$73$ & $[1,1,1,d-1,1,1]$ & $[\frac{1}{2},1,\frac{1}{2},-\frac{3}{2},\frac{1}{2},\frac{1}{2}]$ & $[2d+5,1,1,d]$\\
$84$ & $[1,b-1,0,2,0,1]$ & $[1,-2,-1,2,-1,1]$ & $[2b+7,1,1,0]$\\
$142$ & $[2,b-2,1,1,0,3]$ & $[\frac{3}{2},-2,-\frac{1}{2},1,-\frac{1}{2},\frac{3}{2}]$ & $[1,3,1,b+1]$\\
$142$ & $[2,b-2,1,1,-1,2]$ & $[1,-1,0,0,0,1]$ & $[1,1,3,b-1]$, $[3,1,1,b]$, $b\geq 1$\\
$193$ & $[1,b,1,1,1,1]$ & $[0,-4,2,0,2,0]$ & $[2b+12,0,0,0]$\\
$283$ & $[3,b-4,2,0,3,1]$ & $[1,-3,1,0,1,1]$ & $[0,3,0,b+2]$, $b\geq 1$\\
$283$ & $[3,b-5,2,0,3,1]$ & $[1,-2,\frac{1}{2},0,\frac{1}{2},1]$ & $[2,1,2,b]$, $b\geq 1$\\

$359$ & $[1,b-3,1,1,1,1]$ & $[\frac{1}{2},-\frac{3}{2},\frac{1}{2},0,\frac{1}{2},\frac{1}{2}]$ & $[1,1,3,b-1]$, $[3,1,1,b]$, $b\geq 1$\\
$359$ & $[2,b-9,2,3,2,2]$ & $[\frac{1}{2},-\frac{3}{2},\frac{1}{2},0,\frac{1}{2},\frac{1}{2}]$ & $[1,1,3,b-1]$, $[3,1,1,b]$, $b\geq 1$\\
$359$ & $[1,b,1,1,1,1]$ & $[1,-\frac{9}{2},1,1,1,1]$ & $[0,0,0,b+5]$
\end{tabular}
\label{table-EI-string-part-card5}
\end{table}

\section{The set $\widehat{G_{2(2)}}^d$}\label{sec-G}

This section aims to classify the Dirac series for $G_{2(2)}$, which
 is realized in $\texttt{atlas}$ via the command
\texttt{G:G2\_s}. This equal rank group is centerless, connected, but \emph{not} simply connected.

We adopt the simple roots of $\Delta^+(\frg, \frt_f)$ and  $\Delta^+(\frk, \frt_f)$ as in Knapp \cite[Appendix C]{Kn}.
In particular, its Vogan diagram is presented in Fig.~\ref{Fig-G-Vogan}, where $\alpha_1=(1,-1,0)$ is short, while $\alpha_2=(-2, 1, 1)$ is long. In this case, $\Delta^+(\frg, \frt_f)$ is $G_2$, while $\Delta^+(\frk, \frt_f)$ is $A_1\times A_1$. Indeed,
$\Delta^+(\frk, \frt_f)$ consists of two orthogonal roots: $\gamma_1:=\alpha_1$, $\gamma_2:=3\alpha_1+2\alpha_2$. Let $\xi_1, \xi_2$ (resp., $\varpi_1, \varpi_2$) be the corresponding fundamental weights for $\Delta^+(\frg, \frt_f)$ (resp., $\Delta^+(\frk, \frt_f)$).
We will express the $\texttt{atlas}$ parameters $\lambda$ and $\nu$ in terms of $\xi_1, \xi_2$, and express highest weights of $K$-types in terms of $\varpi_1, \varpi_2$. Note that the $K$-types are parameterized via the highest weight theorem by $[a, b]$ such that $a, b$ are members of $\bbN$ and that $a+b$ is even.

\begin{figure}[H]
\centering
\scalebox{0.6}{\includegraphics{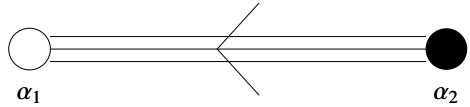}}
\caption{The Vogan diagram for $G_{2(2)}$}
\label{Fig-G-Vogan}
\end{figure}

Carrying out the algorithm in Section \ref{sec-algorithm} for $G_{2(2)}$ leads us to  Table \ref{table-G-scattered-part}, where in the last row sits the trivial representation.

The strings of $\widehat{G_{2(2)}}^d$ are given in Table
\ref{table-G-string-part}.
In this table, the coordinates $a, b$ are members of $\bbN$ such that
the infinitesimal character
$$
\Lambda=[a, b]
$$
and that
\begin{equation}\label{G-common-requirement}
a+b\geq 1.
\end{equation}
In some cases, there are stronger requirements for certain coordinates. They will be put within the column ``spin LKTs". Every spin-lowest $K$-type in these tables occurs with multiplicity one.

\begin{table}[H]
\centering
\caption{FS-scattered members of $\widehat{\rm G_{2(2)}}^{\mathrm{d}}$}
\begin{tabular}{r|c|c|c|c|c|r}
$\# x$ &   $\lambda$  & $\nu$ & spin LKTs & mult & u-small & string limit\\
\hline
$8$ & $[3,0]$ & $[1,0]$ & $ [3,1]$ & $1$ & Yes & $\#4$, $b=-1$\\
$9$ & $[1,1]$ & $[1,0]$ & $[3,1]$ & $1$ & Yes  & $\#3$, $a=-2$\\
$9$ & $[1,1]$ & $[1,1]$ & $[0,0]$ & $1$ & Yes &
\end{tabular}
\label{table-G-scattered-part}
\end{table}

\begin{table}[H]
\centering
\caption{Strings of $\widehat{G_{2(2)}}^{\mathrm{d}}$}
\begin{tabular}{r|c|c|c|r}
$\# x$ &   $\lambda$   & $\nu$ &spin LKTs  \\
\hline
$0$ & $[a,b]$ & $[0,0]$ & $[a+3b+2,a+b]$\\
$1$ & $[a,b]$ & $[0,0]$ & $[2a+3b+3,b-1]$, $b\geq 1$\\
$2$ & $[a,b]$ & $[0,0]$ & $[a-1,a+2b+1]$, $a\geq 1$\\
$3$ & $[a,1]$ & $[-\frac{3}{2},1]$ & $[a+2,a+2]$\\
$4$ & $[1,b]$ & $[1,-\frac{1}{2}]$ & $[3b+4,b]$
\end{tabular}
\label{table-G-string-part}
\end{table}

There are seven \emph{proper} $\theta$-stable parabolic subgroups of $\texttt{G2\_s}$. However, only the following five among them have the form \texttt{Parabolic:(support(x), x)} for certain KGB element $\texttt{x}$:
\begin{verbatim}
([],KGB element #0)
([],KGB element #1)
([],KGB element #2)
([1],KGB element #3)
([0],KGB element #4)
\end{verbatim}
The Levi subgroups of the first three $\theta$-stable parabolic  subgroups are described as
\begin{verbatim}
compact connected quasisplit real group with Lie algebra 'u(1).u(1)'
\end{verbatim}
by \texttt{atlas}, while those of the last two are described as
\begin{verbatim}
connected quasisplit real group with Lie algebra 'sl(2,R).u(1)'
\end{verbatim}
Since we are doing cohomological induction in the way of Theorem \ref{thm-Vogan}, this justifies from another aspect that there are five strings for $\widehat{G_{2(2)}}^{\mathrm{d}}$ in total.

\begin{example}\label{exam-G2-string-starting-point}
The following two representations are scattered members of $\widehat{G_{2(2)}}^{\mathrm{d}}$ in the sense of \cite{D17}. Namely, they can not be cohomologically induced from any \emph{good} module of any proper $\theta$-stable Levi subgroup of $G_{2(2)}$.
\begin{verbatim}
final parameter(x=3,lambda=[0,1]/1,nu=[-3,2]/2)
final parameter(x=4,lambda=[1,0]/1,nu=[2,-1]/2)
\end{verbatim}
However, it is more convenient to include them into the strings in Table \ref{table-G-string-part} with $\#x=3$ and $\#x=4$ respectively as the starting points.
\hfill\qed
\end{example}

\begin{example}\label{exam-G-scattered-via-string}
It is very interesting to note that both of the two non-trivial FS-scattered members in Table \ref{table-G-scattered-part} can be viewed as limits of certain strings.

Firstly, let us input the starting representation of the string with $\#x=3$ in Table \ref{table-G-string-part}.
\begin{verbatim}
G:G2_s
set p=parameter(KGB(G,3), [0,1], [-3/2,1])
set (P, q)=reduce_good_range(p)
q
Value: final parameter(x=2,lambda=[-5,2]/2,nu=[-3,2]/2)
goodness(q,G)
Value: "Weakly good"
\end{verbatim}
The last output says that the inducing module \texttt{q} is weakly good. Now we minus the first coordinate of the lambda parameter of \texttt{q} by $2$, and get \texttt{qm2}, which is no longer weakly good.
\begin{verbatim}
set qm2=parameter(x(q),lambda(q)-[2,0],nu(q))
qm2
Value: final parameter(x=2,lambda=[-9,2]/2,nu=[-3,2]/2)
goodness(qm2, G)
Value: "None"
theta_induce_irreducible(qm2,G)
Value:
1*parameter(x=9,lambda=[1,1]/1,nu=[1,0]/1) [0]
1*parameter(x=6,lambda=[4,-1]/1,nu=[3,-1]/2) [3]
\end{verbatim}
The last output says that the second representation of Table \ref{table-G-scattered-part} occurs as a composition factor of the module cohomologically induced from \texttt{qm2}.
Therefore, we may view  that FS-scattered module as the limit case of the string with $\#x=3$ in Table \ref{table-G-string-part} by taking $a=-2$.

Similarly, we can view the first representation of Table \ref{table-G-scattered-part}
as the limit case of the string with $\#x=4$ in Table \ref{table-G-string-part} by taking $b=-1$.
\hfill\qed
\end{example}
\begin{rmk}
We learned ``string limit" from the referee of \cite{D17E6} and Daniel Wong.
\end{rmk}

\begin{cor}
Let $G$ be \emph{$\texttt{G2\_s}$}. Then all the $K$-types whose spin norm is equal to their lambda norm are exactly
\begin{itemize}
\item[$\bullet$]
$[a+3b+2,a+b]$, where $a, b\geq 0$ and $a+b\geq 1$;
\item[$\bullet$]
$[2a+3b+3,b-1]$, where $a\geq 0$ and $b\geq 1$;
\item[$\bullet$]
$[a-1,a+2b+1]$, where $a\geq 1$ and $b\geq 0$.
\end{itemize}
\end{cor}
\begin{proof}
Note that the strings with $\#x=0, 1, 2$ give precisely all the irreducible tempered representations of $\texttt{G2\_s}$ with non-zero Dirac cohomology. The result follows from Theorem 1.2 of \cite{DD16} and Table \ref{table-G-string-part}.
\end{proof}

Let us illustrate the above corollary in Fig.~\ref{Fig-G-eK}, where the horizontal (resp., vertical) axis gives the $a$-coordinate (resp., the $b$-coordinate) in $a\varpi_1+b\varpi_2$.  We use black dots to stand for those $K$-types whose spin norm equals to their lambda norm, while the other $K$-types are represented by circles. Note that $\|\varpi_2\|=\sqrt{3}\|\varpi_1\|$ and that $a+b$ should be even.

\begin{figure}[H]
\centering
\scalebox{0.6}{\includegraphics{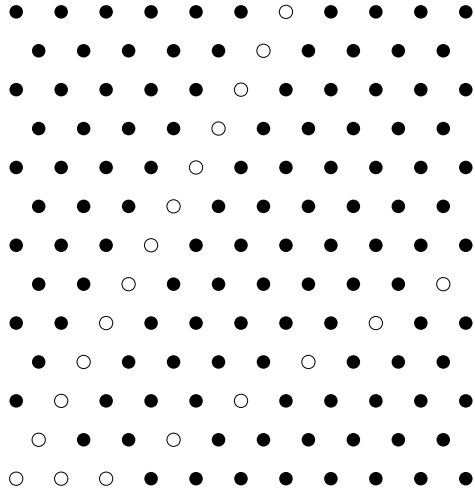}}
\caption{Some $K$-types for $G_{2(2)}$}
\label{Fig-G-eK}
\end{figure}

\medskip
\centerline{\scshape Funding}
Dong was supported by NSFC grant 11571097 (2016-2019).

\medskip
\centerline{\scshape Acknowledgements}
We thank Kei Yuen Chan, Pavle Pand\v zi\'c and Daniel Wong for helpful discussions.
We thank the following \texttt{atlas} mathematicians for their help, and for their  interest in the representations reported in this article: Jeffrey Adams, Steve Miller, Annegret Paul, David Vogan. We thank the referee for very careful reading, and very helpful suggestions.

\end{document}